\documentclass[]{siamart220329}
\usepackage{cite}

\usepackage{lipsum}
\usepackage{amsfonts}
\usepackage{amsmath,booktabs,ctable,threeparttable,boxedminipage}
\usepackage{amssymb,amsfonts,boxedminipage}
\usepackage{graphicx,epstopdf} 

\usepackage[caption=false]{subfig} 
\usepackage{algorithm}
\usepackage{algpseudocode}
\usepackage{qbordermatrix}
\usepackage{blkarray}
\usepackage{lineno}
\usepackage{xcolor}

\ifpdf
  \DeclareGraphicsExtensions{.eps,.pdf,.png,.jpg}
\else
  \DeclareGraphicsExtensions{.eps}
\fi

\newsiamremark{remark}{Remark}
\newsiamremark{hypothesis}{Hypothesis}
\crefname{hypothesis}{Hypothesis}{Hypotheses}
\newsiamthm{claim}{Claim}

\newtheorem{Prop}{Proposition}[section]
\newtheorem{Lem}{lemma}[section]
\newtheorem{Thm}{Theorem}[section]
\newtheorem{Rem}{Remark}[section]

\headers{Preconditioned method for general-form regularization}{H. Li}

\title{A preconditioned Krylov subspace method for linear inverse problems with general-form Tikhonov regularization }

\author{Haibo Li\thanks{Institute of Computing Technology, Chinese Academy of Sciences, 100190 Beijing, China. 
  (\email{haibolee1729@gmail.com}).}
}

\usepackage{amsopn}

\newcommand\transG{*_{G}}
\newcommand\argmin{\mathop{\arg\min}}
\algrenewcommand\algorithmicrequire{\textbf{Input:}}
\algrenewcommand\algorithmicensure{\textbf{Output:}}

\ifpdf
\hypersetup{
  pdftitle={A preconditioned Krylov subspace method for linear inverse problems with general-form Tikhonov regularization},
  pdfauthor={H. Li}
}
\fi

\begin{document}
\maketitle

\begin{abstract}
Tikhonov regularization is a widely used technique in solving inverse problems that can enforce prior properties on the desired solution. In this paper, we propose a Krylov subspace based iterative method for solving linear inverse problems with general-form Tikhonov regularization term $x^TMx$, where $M$ is a positive semi-definite matrix. An iterative process called the preconditioned Golub-Kahan bidiagonalization (pGKB) is designed, which implicitly utilizes a proper preconditioner to generate a series of solution subspaces with desirable properties encoded by the regularizer $x^TMx$. Based on the pGKB process, we propose an iterative regularization algorithm via projecting the original problem onto small dimensional solution subspaces. We analyze regularization effect of this algorithm, including the incorporation of prior properties of the desired solution into the solution subspace and the semi-convergence behavior of regularized solution. To overcome instabilities caused by semi-convergence, we further propose two pGKB based hybrid regularization algorithms. All the proposed algorithms are tested on both small-scale and large-scale linear inverse problems. Numerical results demonstrate that these iterative algorithms exhibit excellent performance, outperforming other state-of-the-art algorithms in some cases.
\end{abstract}

\begin{keywords}
inverse problems, ill-posed, general-form Tikhonov regularization, preconditioned Golub-Kahan bidiagonalization, subspace projection regularization, hybrid regularization
\end{keywords}

\begin{MSCcodes}
65F22, 65J20, 65J22
\end{MSCcodes}

\section{Introduction}
Inverse problems arise in various fields of science and engineering, where the aim is to recover unknown parameters or functions from observed data. Such problems are often encountered in many applications, including image reconstruction, medical imaging, geophysics, data assimilation and so on \cite{Kaip2006,Hansen2006,Buzug2008,Law2015,Richter2016}. Formally, a linear inverse problem after discretization leads to the following linear system:
\begin{equation}\label{inverse1}
	A x_{\text{true}} + e = b
\end{equation}
where $A\in\mathbb{R}^{m\times n}$ is the (discretized) forward operator that maps the unknown quantity to the observed data, $e$ is the noise in the observed data, and $x_{\text{true}}$ is the underlying quantity we wish to reconstruct. One key issue with inverse problems is that they are usually ill-posed. For \cref{inverse1} it means that $A$ is extremely ill-conditioned such that small perturbation in $b$ leads to large changes in the solution, or $A$ is under-determined such that there may be multiple solutions that fit the data equally well \cite{Hansen1998,Engl2000}. These difficulties stem from the fact that the inverse of the forward operator is usually discontinuous or fails to preserve certain properties of the desired solution, such as smoothness or sparsity \cite{Engl2000}. 

To overcome these challenges, regularization techniques are commonly employed, which uses prior knowledge about the underlying solution to constrain the set of possible solutions and improve their stability and uniqueness properties. The idea is to introduce a penalty term into the objective function to promote solutions that are smooth, sparse, or have other desirable properties \cite{Engl2000,Hansen2010}. It seeks to solve the following regularized inverse problem:
\begin{equation}\label{Tikhon1}
x_{\lambda} = \argmin_{x \in \mathbb{R}^{n}}\{\|Ax-b\|_{2}^{2} + \lambda R(x)\},
\end{equation}
where $\lambda>0$ is a regularization parameter that controls the trade-off between data fit and regularization, and $R(x)$ is a regularizer that encodes our prior knowledge about the solution. Tikhonov regularization is a popular choice for regularizing ill-posed inverse problems, which corresponds to choosing the regularizer $R(x)=\|Lx\|_{2}^{2}$ for some linear operator $L\in\mathbb{R}^{p\times n}$ that maps $x$ to a suitable space \cite{Tikhonov1977}. Regularization based on Bayesian inference is another commonly used regularization method \cite{Kaip2006,Stuart2010}. Suppose $e\sim\mathcal{N}(0,\mu^{-1}I)$ is a Gaussian random vector. In the paper we denote by $I$ the identity matrix with orders clear from the context. Then the likelihood of $b|x$ satisfies $\pi(b|x) \propto \exp\left(-\frac{\mu}{2}(Ax-b)^T(Ax-b)\right)$. If we choose a Gaussian prior $\pi(x) \propto \exp\left(-\frac{\sigma}{2}x^{T}Mx\right)$ to model the distribution of $x$, where $M$ is positive semidefinite, by the Bayes' formula we have the posterior likelihood:
\begin{align*}
	\pi(x|b) 
	\propto \pi(x)\pi(b|x)
	\propto \exp\left(-\frac{\mu}{2}(Ax-b)^T(Ax-b)-\frac{\sigma}{2}x^{T}Mx\right).
\end{align*}
By neglecting the scaling factor $\mu/2$, the maximum a posterior (MAP) estimate of $x$ is the solution to
\begin{equation}\label{Bayes1}
	\min_{x \in \mathbb{R}^{n}}\{\|Ax-b\|_{2}^{2} + \frac{\sigma}{\mu}x^TMx\}.
\end{equation}
Comparing \cref{Bayes1} with \cref{Tikhon1}, we know that $\sigma/\mu$ plays the role of the regularization parameter $\lambda$, and if $M=L^{T}L$ is the square root decomposition of $M$ (note in some literature the square root of $M$ requires $L^T=L$), then $x^{T}Mx=\|Lx\|_{2}^{2}$ is just the Tikhonov regularization term which comes from the prior distribution of $x$ that encodes the structure we expect to enforce on $x$.

In this paper, we consider iterative methods for solving the two equivalent regularized inverse problems
\begin{equation}\label{regular1}
	\min_{x \in \mathbb{R}^{n}}\{\|Ax-b\|_{2}^{2} + \lambda x^TMx\} \ \  \mathrm{or} \ \ 
	\min_{x \in \mathbb{R}^{n}}\{\|Ax-b\|_{2}^{2} + \lambda \|Lx\|_{2}^{2}\}
\end{equation}
within the subspace projection regularization framework
\begin{equation}\label{subspace_regu1}
	\min_{x\in\mathcal{X}_k}x^TMx, \ \ \mathcal{X}_k = \{x: \min_{x\in\mathcal{S}_k}\|Ax-b\|_{2}\} 
\end{equation}
to avoid choosing in advance regularization parameters. A series of solution subspaces $\mathcal{S}_k\subseteq\mathbb{R}^{n}$ of dimension $k=1,2,\dots$ should be constructed to incorporate prior properties of the solution encoded by the regularizer $x^TMx$. For standard-form regularization with $M = I$, the most popular iterative regularization method is LSQR \cite{Paige1982} with an early stopping rule, which projects \cref{inverse1} onto a sequence of lower dimensional Krylov subspaces to approximate $x_{\text{true}}$ \cite{Oleary1981,Bjorck1988}. The iteration should stop early to overcome semi-convergence by some criteria such as L-curve, discrepancy principle or GCV \cite{Morozov1966,Golub1979,Hansen1992}. For general-form regularization that $M\neq I$, if $L$ is already available or the decomposition $M=L^{T}L$ can be obtained without too much computation, there are many methods deal with the regularizer $\|Lx\|_{2}^{2}$ instead of $x^TMx$. For an invertible $L$, we can use $L$ as a preconditioner by the substitution $y = Lx$; see \cite{Calvetti2005priorconditioners,Calvetti2007,Calvetti2018} for preconditioned methods for Bayesian inverse problems. Otherwise, \cref{regular1} can be transformed to the standard-form by using the $A$-weighted pseudoinverse $L_{A}^{\dag}$; see \cite{Elden1982} or \cite[\S 2.3]{Hansen1998} for details. However, such transformations are often computationally unfeasible for large-scale matrices. For large-scale $A$ and $L$, there are some other iterative regularization methods, such as the modified truncated SVD method \cite{Hansen1992modified,Bai2021novel,Huang2022tikhonov}, joint bidiagonalization method \cite{Kilmer2007Projection},
methods based on randomized GSVD of $\{A,L\}$ \cite{Xiang2015,Wei2016,Vatankhah2018total}, methods based on generalized Krylov subspace \cite{Lampe2012large,Reichel2012tikhonov,Huang2016projected} and so on.

However, all the above methods require $L$, which is not available in some scenarios of applications. For example, for the Met\'{e}rn class of covariance functions describing the prior of $x$, which are also called Met\'{e}rn kernels \cite{Genton2001classes,Roininen2011correlation}, the corresponding $M$ can be large-scale and dense and thus computing $L$ is extremely expensive. Another class of frequently encountered examples is arisen in the lagged diffusivity fixed point (LDFP) iteration method for nonlinear regularizers \cite{Vogel1996iterative,Chan1999convergence}, such as total variation \cite{Rudin1992nonlinear} or Perona-Malik \cite{Perona1990scale} regularizer used in image reconstruction and electrical impedance tomography, where at each outer iteration of LDFP, a large-scale $M$ is constructed to linearize the nonlinear regularizer and a corresponding regularization problem \cref{regular1} needs to be solved. For these cases, we have to deal with $x^TMx$ instead of $\|Lx\|_{2}^{2}$. If $M$ is positive definite and $M^{-1}$ (i.e., covariance matrix of the prior) is already known, which is often the case for Met\'{e}rn kernels, the generalized Golub-Kahan bidiagonalization method \cite{Chung2017} is very efficient. For many cases that $M^{-1}$ is unknown, a preconditioned LSQR method called MLSQR has been proposed where a linear system $Mx=y$ needs to be iteratively solved at each iteration \cite{Arridge2014}, and it has been used in the inner iteration of LDFP for many applications; see e.g. \cite{Harhanen2015,Hannukainen2016edge,Bin2020irn}. 
To successfully applying MLSQR, a big challenge is that $M$ is often non-invertible or nearly singular. Even if $M$ is invertible, it is often the case that $M$ has a very large condition number, which results in that too many iterations are needed for solving $Mx=y$, thus significantly reduces the efficiency of this algorithm. In some work such as \cite{Hannukainen2016,Candiani2021}, the researchers suggest to replace $M$ by $M_{\delta}=M+\delta I$ to make $M_{\delta}$ positive definite and well-conditioned. But the proper value of $\delta$ can only be set by numerical trials, and there may be an accuracy sacrifice of the solution since the target problem has been changed. 

In this paper, we propose a new Krylov subspace based regularization method to deal with the regularizer $x^TMx$ for positive semidefinite matrix $M$, and our method do not need replace $M$ by a positive definite $M_{\delta}$. To this end, we first design an  iterative process that generates a series of vectors spanning the solution subspaces, where a proper preconditioner is implicitly constructed and exploited. This preconditioner is proper in the sense that the generated solution subspaces can incorporate prior properties of the desired solution encoded by the regularizer $x^TMx$, resulting in an iterative regularized solution of high quality. The main contributions of this paper are listed as follows:
\begin{itemize}
	\item We design an iterative process similar to the Golub-Kahan bidiagonalization (GKB) to generates a series of solution subspaces. This process is proved to be mathematically equivalent to the standard GKB process of a preconditioned $A$ where a right preconditioner is implicitly used, thereby we name it as the preconditioned GKB (pGKB) process.
	\item  By giving explicit expression of the solution subspace using the GSVD of $\{A,L\}$, we show that the above solution subspace can incorporate prior properties of the desired solution encoded by the regularizer. Based on the pGKB process, we propose a subspace projection regularization algorithm via projecting the original problem onto the solution subspace at each iteration. 
	\item We analyze regularization effect of the above proposed algorithm. We prove that the iterative solution has a filtered GSVD expansion form, where some dominant GSVD components are captured and the others are filtered out. This reveals that the algorithm exhibits typical semi-convergence behavior where the iteration number $k$ plays the role of regularization parameter.
	\item To overcome instabilities arising from semi-convergence, two pGKB based hybrid regularization algorithms are proposed, where Tikhonov regularization is applied to the projected small-scale problem at each iteration. To efficiently determine regularization parameters for the projected small-scale problems, the weighted GCV (WGCV) method and ``secant update'' method based on the discrepancy principle are adopted, respectively.
\end{itemize}
All these proposed iterative algorithms are tested on both small-scale and large-scale linear inverse problems to show excellent effectiveness and performance.

The paper is organized as follows. In \Cref{sec2}, we review basic properties of the general-form Tikhonov regularization using the GSVD of $\{A,L\}$. In \Cref{sec3}, we design the pGKB process and propose the pGKB based subspace projection regularization (pGKB\_SPR) algorithm. In \Cref{sec4}, we analyze regularization effect of pGKB\_SPR and reveal the semi-convergence behavior of it. To overcome instabilities caused by semi-convergence, in \Cref{sec5} we propose two pGKB based hybrid regularization algorithms. In \Cref{sec6}, we choose several small-scale and large-scale linear inverse problems to test the proposed algorithms. Finally, we give some concluding remarks in \Cref{sec7}.

Throughout the paper, we denote by $I$ and $\boldsymbol{0}$ the identity matrix and zero matrix/vector, respectively, with orders clear from the context, and denote by $\mathrm{span}\{\cdot\}$ the subspace spanned by a group of vectors or columns of a matrix.

\section{General-form Tikhonov regularization and GSVD}\label{sec2}
Although in this paper, the square root decomposition $M=L^TL$ is not needed, it is convenient to use $L$ to make some analysis. Since $M$ is positive semidefinite, we write $L$ in the compact form, i.e., $L\in\mathbb{R}^{p\times n}$ with $p\leq n$ and $\mathrm{rank}(L)=p$. Taking the gradient of $x$ in \cref{regular1} leads to
\begin{equation}\label{normal}
	(A^TA+\lambda M)x = A^Tb .
\end{equation}
Denote by $\mathcal{N}(\cdot)$ the null space of a matrix. In order to ensure that there exist a unique regularized solution for any $\lambda>0$, the sufficient and necessary condition is
\begin{equation}\label{rank1}
	\mathrm{rank}(A^TA+\lambda M)=n \ \Longleftrightarrow \ \mathcal{N}(A)\cap \mathcal{N}(L)=\{\boldsymbol{0}\},
\end{equation}
and the solution is
\begin{equation}\label{Tihk_solut}
	x_{\lambda} = (A^TA+\lambda M)^{-1}A^Tb .
\end{equation}
This solution can be expressed in a more convenient form using the GSVD of $\{A,L\}$ \cite{Van1976}, which is given by 
\begin{equation}\label{gsvd}
	A = U_AD_AZ^{-1}, \ \ 
	L = U_LD_LZ^{-1},
\end{equation}
where $U_A\in\mathbb{R}^{m\times m}$ and $U_L\in\mathbb{R}^{p\times p}$ are orthogonal, $Z\in\mathbb{R}^{n\times n}$ is invertible, and
\begin{equation*}
	D_A = \bordermatrix*[()]{%
	I_{A}  &  &  & r \cr
	&  \Sigma_{A}  &  & q \cr
	&  & O_{A}  & m-r-q \cr
	r & q & n-r-q
} , \ 
D_L =
\bordermatrix*[()]{%
	O_L  &  &  & p+r-n \cr
	&  \Sigma_{L}  &  & q \cr
	&  & I_{L}  & n-r-q \cr
	r & q & n-r-q
} ,
\end{equation*}
with identity matrices $I_A$ and $I_{L}$, zero matrices $O_{A}$ and $O_{L}$, and diagonal matrices $\Sigma_{A}=\mathrm{diag}(\sigma_{r+1},\dots,\sigma_{r+q})$ and $\Sigma_{L}=\mathrm{diag}(\rho_{r+1},\dots,\rho_{r+q})$. The identity $\Sigma_{A}^{T}\Sigma_{A}+\Sigma_{L}^{T}\Sigma_{L}=I$ holds. If we arrange $\sigma_i$ and $\rho_i$ in decreasing order of values such that
\begin{equation}
	1>\sigma_{r+1}\geq\sigma_{r+2}\geq\cdots\geq\sigma_{r+q}>0, \ \ \ 
	0<\rho_{r+1}\leq\rho_{r+2}\leq\cdots\leq\rho_{r+q}<1 ,
\end{equation}
then $\sigma_{i}^{2}+\rho_{i}^{2}=1$, and $\gamma_i:=\sigma_i/\rho_i$ is called the $i$-th generalized singular value of $\{A,L\}$. We define $\gamma_{i}=\infty$ for $1\leq i\leq r$ and $\gamma_i=0$ for $r+q+1\leq i\leq n$. 

Denote the columns of $U_A$ and $Z$ by $\{u_{A,i}\}_{i=1}^{m}$ and $\{z_i\}_{i=1}^{n}$, respectively. The discrete Picard condition (DPC) plays a central role in the regularization of discrete ill-posed problems, which says that the Fourier coefficients $|u_{A,i}^T b_{\text{true}}|$ on average decay to zero faster than the corresponding $\gamma_{i}$, where $b_{\text{true}}=Ax_{\text{true}}$. Any regularization is based on an underlying requirement that the DPC is satisfied, only under which can one compute a regularized solution with some accuracy \cite[\S 4.5]{Hansen1998}. 
Using the decomposition \cref{gsvd}, the Tikhonov solution $x_{\lambda}$ can be written as 
\begin{equation}\label{Tihk_gsvd}
	x_{\lambda} = \sum_{i=1}^{r}(u_{A,i}^{T}b)z_i + \sum_{i=r+1}^{r+q}\frac{\gamma_{i}^{2}}{\gamma_{i}^{2}+\lambda}\frac{u_{A,i}^{T}b}{\sigma_i}z_i.
\end{equation}
The factors $f_{i}:=\gamma_{i}^{2}/(\gamma_{i}^{2}+\lambda)$ can be viewed as filters applied to noisy coefficients $u_{A,i}^{T}b$. A proper value of $\lambda$ should satisfy that $f_{i}\approx 1$ for small $i$ and $f_{i}\approx 0$ for large $i$, and thereby dampen noisy components appeared in the regularized solution. Although there are several methods for choosing regularization parameter, one big challenge is that in order to find a suitable $\lambda$, many different values of $\lambda$ must be tried to solve \eqref{regular1} in advance or the GSVD of $\{A,L\}$ should be computed, which is computationally expensive for large-scale problems. 

For large-scale problems, an alternative is the subspace projection regularization method \cite[\S 3.3]{Engl2000}, which seeks to compute a series of $x_k$ as the solution to
\begin{equation}\label{subspace_regu}
	\min_{x\in\mathcal{X}_k}x^TMx, \ \ \mathcal{X}_k = \{x: \min_{x\in\mathcal{S}_k}\|Ax-b\|_{2}\} 
\end{equation}
where $\mathcal{S}_k$ is the subspace of $\mathbb{R}^{n}$ of dimension $k=1,2,\dots$, and the iteration proceeds until an early stopping criterion is satisfied to overcome under-regularizing caused by semi-convergence. 

\begin{Rem}
	Under the condition \cref{rank1}, the subspace projected problem \cref{subspace_regu} has a unique solution for $k=1,\dots,n$. Let $S_k$ be an $n\times k$ column orthonormal matrix whose columns span $\mathcal{S}_k$. By writing any $x\in\mathcal{S}_k$ as $x=S_ky$, we get the solution to \cref{subspace_regu} as $x_k=S_ky_k$, where $y_k$ is the solution to $\min_{x\in\mathcal{Y}_k}\|LS_ky\|_2, \ \mathcal{Y}_k = \{y: \min_{y\in\mathbb{R}^{k}}\|AS_ky-b\|_{2}\}$. By \cite[Theorem 2.1]{Elden1982}, there exist a unique solution $y_k$ if and only if $\mathcal{N}(AS_k)\cap\mathcal{N}(LS_k)=\{\boldsymbol{0}\}$, which is equivalent to $\mathrm{rank}(\begin{pmatrix}
		A \\ L
	\end{pmatrix}S_k)=n$. Since $S_k$ has full column rank, it means that $\mathrm{rank}(\begin{pmatrix}
		A \\ L
	\end{pmatrix})=n \Leftrightarrow \mathcal{N}(A)\cap\mathcal{N}(L)=\{\boldsymbol{0}\}$. Thus \cref{rank1} is the sufficient and necessary conditioned for that \cref{subspace_regu} has a unique solution.
\end{Rem}

We call $\mathcal{S}_k$ the \textit{solution subspace}, which should be constructed elaborately such that the prior information about the desired solution is incorporated in $\mathcal{S}_k$. For the general-form regularizer, from the expression \cref{Tihk_gsvd} we know that the most ideal choice is $\mathcal{S}_k=\mathrm{span}\{Z_k\}$, where $Z_k=(z_1,\dots, z_k)$. This leads to the truncated GSVD (TGSVD) solution 
\begin{equation}\label{2.17}
x_{k}^{\text{TGSVD}} = \sum_{i=1}^{k}\frac{u_{A,i}^Tb}{\sigma_i}z_i,
\end{equation}
where we let $\sigma_i=1$ for $1\leq i\leq r$ \cite[\S 3.2]{Hansen1998}. By truncating the GSVD at a proper $k$, the obtained solution can capture main information corresponding to dominant right generalized singular vectors while suppress noise corresponding to others. For large-scale problems that GSVD is computationally expensive, a legitimate subspace projection regularization method should construct a series of solution subspaces that can approximate well dominant right generalized singular vectors $z_i$ as the iterations progress. To this end, we first gain an insight into the relation between $z_i$ and the matrix pair $\{A,M\}$.

\begin{Prop}\label{prop2.1}
	Let $G=A^TA+\alpha M$ with $\alpha>0$. The generalized eigenvalues of $A^TAz=\xi Gz$ in decreasing order are
	\begin{equation}\label{gen_eigen}
		\underbrace{1,\dots,1}_{r},\underbrace{\gamma_{r+1}^2/(\gamma_{r+1}^2+\alpha),\dots,\gamma_{r+q}^2/(\gamma_{r+q}^2+\alpha)}_{q},\underbrace{0,\dots,0}_{n-r-q} ,
	\end{equation}
	and the corresponding generalized eigenvectors 
	are $\{z_i\}_{i=1}^{n}$.
\end{Prop}
\begin{proof}
	Observing from \cref{gsvd} that
	\begin{equation*}
		A^TA = Z^{-T}D_{A}^TD_AZ^{-1}, \ \ \ M = L^TL = Z^{-T}D_{L}^TD_LZ^{-1},
	\end{equation*}
	where $D_{A}^TD_A$ and $D_{L}^TD_L$ are diagonal matrices of the following form:
	\begin{align*}
		& D_{A}^TD_A = \mathrm{diag}(\underbrace{1,\dots,1}_{r},\underbrace{\sigma_{r+1}^2,\dots,\sigma_{r+q}^2}_{q},\underbrace{0,\dots,0}_{n-r-q}), \\
		& D_{L}^TD_L = \mathrm{diag}(\underbrace{0,\dots,0}_{r},\underbrace{\rho_{r+1}^2,\dots,\rho_{r+q}^2}_{q},\underbrace{1,\dots,1}_{n-r-q}) ,
	\end{align*}
	we have $G=A^TA+\alpha M = Z^{-T}D_{\alpha}Z^{-1}$ with
	\begin{equation}\label{D_alpha}
		D_{\alpha} = \mathrm{diag}(\underbrace{1,\dots,1}_{r},\underbrace{\sigma_{r+1}^2+\alpha\rho_{r+1}^2,\dots,\sigma_{r+q}^2+\alpha\rho_{r+q}^2}_{q},\underbrace{\alpha,\dots,\alpha}_{n-r-q}) .
	\end{equation}
	Therefore, $\{z_i\}_{i=1}^{n}$ are generalized eigenvectors of $A^TAz=\xi Gz$ with generalized eigenvalues the diagonals of $D_{A}^TD_{A}D_{\alpha}^{-1}$, which are
	\begin{align*}
		& \ \ \ \underbrace{1,\dots,1}_{r},\underbrace{\sigma_{r+1}^2/(\sigma_{r+1}^2+\alpha\rho_{r+1}^2),\dots,\sigma_{r+q}^2/(\sigma_{r+q}^2+\alpha\rho_{r+q}^2)}_{q},\underbrace{0,\dots,0}_{n-r-q} \\
		&= \underbrace{1,\dots,1}_{r},\underbrace{\gamma_{r+1}^2/(\gamma_{r+1}^2+\alpha),\dots,\gamma_{r+q}^2/(\gamma_{r+q}^2+\alpha)}_{q},\underbrace{0,\dots,0}_{n-r-q} .
	\end{align*}
	Note that $\gamma_{i}^2/(\gamma_{i}^2+\alpha)$ increases with respect to $\gamma_i$. The desired result is obtained.
\end{proof}



Note that $G$ is positive definite. Inspired by \Cref{prop2.1}, we consider constructing solution spaces in the $G$-inner product space, where the $G$-inner product in $\mathbb{R}^{n}$ is defined by $\langle x, x' \rangle_G = x^TGx'$ for any $x,x'\in \mathbb{R}^{n}$. We hope those dominant $z_i$ can be well captured by the solution subspaces so that the iterative solution will incorporate main features described by those $z_i$.

\section{Preconditioned Golub-Kahan bidiagonalization and subspace projection regularization algorithm}\label{sec3}
For standard-form regularization, the most popular LSQR algorithm bases upon the Golub-Kahan bidiagonalization (GKB) that constructs solution subspaces in the standard inner product space $\mathbb{R}^{n}$. In order to construct proper solution subspaces for regularizer $x^TMx$, we consider the GKB process using the $G$-inner product. Note that the standard GKB process needs $A$ and its transpose $A^T$. For the $G$-inner product case, we need the matrix-form expression of the adjoint of a linear operator between the $G$- and standard inner product Hilbert spaces $(\mathbb{R}^{n}, \langle \cdot , \cdot \rangle_G)$ and $(\mathbb{R}^{m}, \langle \cdot , \cdot \rangle_2)$.

\begin{Lem}\label{lem3.1}
	For the linear operator $A: (\mathbb{R}^{n}, \langle \cdot , \cdot \rangle_G) \to (\mathbb{R}^{m}, \langle \cdot , \cdot \rangle_2)$ between the two Hilbert spaces, define $A^{\transG}: (\mathbb{R}^{m}, \langle \cdot , \cdot \rangle_2) \to (\mathbb{R}^{n}, \langle \cdot , \cdot \rangle_G)$, which is the adjoint of $A$, by $\langle Ax, y \rangle_2=\langle x, A^{\transG}y \rangle_{G}$ for any $x\in\mathbb{R}^{n}$ and $y\in\mathbb{R}^{m}$. Then the matrix-form of $A^{\transG}$ is
	\begin{equation}\label{transG}
		A^{\transG} = G^{-1}A^{T}.
	\end{equation}
\end{Lem}
\begin{proof}
	First note that $A^{\transG}$ is well-defined, which can be found in any functional analysis textbook. Since $\langle Ax, y \rangle_{2}=x^TA^Ty$ and $\langle x, A^{\transG}y \rangle_{G}=x^TGA^{\transG}y$. Let $x^TA^Ty=x^TGA^{\transG}y$ for any vectors $x$ and $y$. It must holds that $A^T=GA^{\transG}$, and we immediately obtain \cref{transG}.
\end{proof}

For the standard inner product in $\mathbb{R}^{n}$, i.e. $G=I$, we have $A^{\transG}=A^T$, which is just the standard matrix transpose. Now we seek to construct solution subspaces $\mathcal{S}_k\subseteq (\mathbb{R}^{n}, \langle \cdot , \cdot \rangle_G)$ to iteratively solve $\min_{x\in\mathcal{S}_k}\|Ax-b\|_{2}$. With the help of \Cref{lem3.1}, the construction of $\mathcal{S}_k$ can be done by the GKB process of $A$ with starting vector $b$ between the two Hilbert spaces $(\mathbb{R}^{n}, \langle \cdot , \cdot \rangle_G)$ and $(\mathbb{R}^{m}, \langle \cdot , \cdot \rangle_2)$. This process can be written as the following recursive relations:
\begin{align}
	& \beta_1u_1 = b, \ \ \alpha_1w_1 = A^{\transG}u_1,  \label{GKB1} \\
	& \beta_{i+1}u_{i+1} = Aw_i - \alpha_iu_i, \label{GKB2} \\
	& \alpha_{i+1}w_{i+1} = A^{\transG}u_{i+1} - \beta_{i+1}w_i, \label{GKB3}
\end{align}
where $u_i\in(\mathbb{R}^{m}, \langle \cdot , \cdot \rangle_2)$, $w_i\in(\mathbb{R}^{n}, \langle \cdot , \cdot \rangle_G)$, and $\alpha_i$, $\beta_i$ should be computed such that $\|u_i\|_2=\|w_{i}\|_G=1$. Thus we have $u_1=b/\beta_1$ with $\beta_1=\|b\|_2$. Using the matrix-form expression of $A^{\transG}$ we have
\begin{align}\label{GKB32}
	\alpha_{i+1}w_{i+1}
	= G^{-1}A^Tu_{i+1}-\beta_{i+1}w_i 
\end{align}
with $\alpha_{i+1} = \|G^{-1}A^Tu_{i+1}-\beta_{i+1}w_{i}\|_G$. Note that for $i=0$ we define $w_0=\boldsymbol{0}$. 

Based on the coupled recursive relations \cref{GKB2,GKB32}, we get the preconditioned Golub-Kahan bidiagonalization (pGKB) process , which is summarized in \Cref{alg1}. The usage of the name ``preconditioned'' will be explained later.

\begin{algorithm}[htb]
	\caption{Preconditioned Golub-Kahan bidiagonalization (pGKB)}
	\begin{algorithmic}[1]\label{alg1}
		\Require $A\in\mathbb{R}^{m\times n}$, $b\in\mathbb{R}^{m}$, $M\in\mathbb{R}^{n\times n}$, $\alpha>0$
		\State $G=A^TA+\alpha M$
		\State $\beta_1=\|b\|_2$, \ $u_1=b/\beta_1$
		\State Compute $s$ by solving $Gs=A^Tu_1$
		\State $\alpha_1 = (s^TGs)^{1/2}$, \ $w_1=s/\alpha_1$
		\For {$i=1,2,\dots,k,$}
		\State $r=Aw_i-\alpha_iu_i$
		\State $\beta_{i+1}=\|r\|_2$, \ $u_{i+1}=r/\beta_{i+1}$
		\State Compute $s$ by solving $Gs=A^Tu_{i+1}$
		\State $s=s-\beta_{i+1}w_i$
		\State $\alpha_{i+1}= (s^TGs)^{1/2}$, \ $w_{i+1}=s/\alpha_{i+1}$
		\EndFor
		\Ensure $\{\alpha_i, \beta_i\}_{i+1}^{k+1}$, \ $\{u_i, w_i\}_{i+1}^{k+1}$
	\end{algorithmic}
\end{algorithm}

After $k$ steps, the pGKB generates two groups of vectors $\{u_i\}_{i=1}^{k+1}$ and $\{w_i\}_{i=1}^{k+1}$. Define two matrices as $U_k=(u_1,\dots,u_k)$ and $W_{k}=(w_1,\dots,w_k)$. Then by \cref{GKB1,GKB2,GKB32}, we have the matrix-form recursive relations:
\begin{align}
	& \beta_1U_{k+1}e_{1} = b, \label{GKB13} \\
	& AW_k = U_{k+1}B_k, \label{GKB23} \\
	& G^{-1}A^TU_{k+1} = W_kB_{k}^{T}+\alpha_{k+1}w_{k+1}e_{k+1}^T, \label{GKB33}
\end{align}
where
\begin{equation}
	B_{k}=\begin{pmatrix}
		\alpha_{1} & & & \\
		\beta_{2} &\alpha_{2} & & \\
		&\beta_{3} &\ddots & \\
		& &\ddots &\alpha_{k} \\
		& & &\beta_{k+1}
		\end{pmatrix}\in  \mathbb{R}^{(k+1)\times k}
\end{equation}
and $e_1$ and $e_{k+1}$ are the first and $(k+1)$-th columns of the identity matrix of order $k+1$, respectively. In fact, $U_{k+1}$ is a 2-orthonormal matrix and $W_k$ is a $G$-orthonormal matrix, which will be proved in the following lemma. Thus by \cref{GKB23} we have $B_k=U_{k+1}^TAW_k$, which implies that $B_k$ is a projection of $A$ onto two subspaces $\mathrm{span}\{U_{k+1}\}$ and $\mathrm{span}\{W_k\}$. Now we analyze the structure of these two subspaces.

\begin{Lem}\label{prop3.1}
	The group of vectors $\{u_i\}_{i=1}^{k}$ is a 2-orthonormal basis of the Krylov subspace 
	\begin{equation}
		\mathcal{K}_k(AG^{-1}A^T, b) = \mathrm{span}\{(AG^{-1}A^T)^ib\}_{i=0}^{k-1},
	\end{equation}
	and $\{w_i\}_{i=1}^{k}$ is a $G$-orthonormal basis of the Krylov subspace 
	\begin{equation}
		\mathcal{K}_k(G^{-1}A^TA, G^{-1}A^Tb) = \mathrm{span}\{(G^{-1}A^TA)^iG^{-1}A^Tb\}_{i=0}^{k-1}.
	\end{equation}
\end{Lem}
\begin{proof}
	In order to get more insights into the pGKB process, we give two proofs. 

	\textit{The first proof.} This proof exploits the property of GKB process of $A$ between the two Hilbert spaces $(\mathbb{R}^{n}, \langle \cdot , \cdot \rangle_G)$ and $(\mathbb{R}^{m}, \langle \cdot , \cdot \rangle_2)$, which states that $\{u_i\}_{i=1}^{k}$ and $\{z_i\}_{i=1}^{k}$ are the 2-orthonormal basis and $G$-orthonormal basis of the Krylov subspaces $\mathcal{K}_k(AA^{\transG}, b)$ and $\mathcal{K}_k(A^{\transG}A, A^{\transG}b)$, respectively. Note that
	$AA^{\transG}=AG^{-1}A^T$, $A^{\transG}A=G^{-1}A^TA$ and $A^{\transG}b=G^{-1}A^Tb$. The proof is completed.

	\textit{The second proof.} Suppose the Cholesky factorization of $G$ is $G=R^TR$. Let $\bar{A}=AR^{-1}$, $v_i=Rw_i$ and $V_k=(v_1,\dots,v_k)$. Then by \cref{GKB23,GKB33} we have
	\begin{equation}\label{GKB4}
		\bar{A}V_k = U_{k+1}B_k, \ \ \ 
		\bar{A}^TU_{k+1} = V_{k}B_{k}^{T}+\alpha_{k+1}v_{k+1}e_{k+1}^{T}.
	\end{equation}
	Since $\|w_i\|_G=1$, we have $\|v_i\|_2=(w_{i}R^{T}Rw_{i})^{1/2}=(w_{i}^{T}Gw_i)^{1/2}=1$. Combining relations \cref{GKB13,GKB4} and using $\|u_i\|_2=\|v_{i}\|_2$, we know that $\{u_i\}_{i=1}^{k}$ and $\{v_i\}_{i=1}^{k}$ are the Lanczos vectors generated by the GKB process of $\bar{A}$ with starting vector $b$ under the standard inner product. Therefore, $\{u_i\}_{i=1}^{k}$ and $\{v_i\}_{i=1}^{k}$ are 2-orthonormal bases of Krylov subspaces $\mathcal{K}_k(\bar{A}\bar{A}^T, b)=\mathcal{K}_k(AG^{-1}A^T, b)$ and $\mathcal{K}_k(\bar{A}^T\bar{A}, \bar{A}^Tb)=\mathcal{K}_k(R^{-T}A^TAR^{-1}, R^{-T}A^Tb)$, respectively. Using $w_i=R^{-1}v_i$, we obtain that $\{w_i\}_{i=1}^{k}$ is a $G$-orthonormal basis of the subspace 
	\[R^{-1}\mathcal{K}_k(\bar{A}^T\bar{A}, \bar{A}^Tb)= R^{-1}\mathrm{span}\{(R^{-T}A^TAR^{-1})^iR^{-T}A^Tb\}_{i=0}^{k-1}.\]
	Note that
	\begin{align*}
		R^{-1}(R^{-T}A^TAR^{-1})^iR^{-T}A^Tb
		&= R^{-1}R^{-T}A^T(AR^{-1}R^{-T}A^T)^ib \\
		&= G^{-1}A^T(AG^{-1}A^T)^ib \\
		&= (G^{-1}A^TA)^iG^{-1}A^Tb .
	\end{align*}
	The proof is completed.
\end{proof}

The second proof explains why we use the name ``preconditioned'' for \Cref{alg1}. The pGKB is essentially the standard GKB process of the preconditioned matrix $AR^{-1}$ with starting vector $b$, where $R^{-1}$ is the right preconditioner. It reduces $AR^{-1}$ to a bidiagonal form $B_k$ while generating two orthonormal matrices $U_{k+1}$ and $V_{k+1}$, and the desired vectors $w_i$ are obtained by $w_i=R^{-1}v_i$. In \Cref{alg1}, the Cholesky factor $R$ and $R^{-1}$ need not to be explicitly computed to get $w_i$, while instead a linear system $Gs=A^Tu_i$ must be solved.

Based on the pGKB process, we propose the subspace projection regularization algorithm by letting $\mathcal{S}_k=\mathrm{span}\{W_k\}$ and solving \cref{subspace_regu}. Suppose the $k$-step pGKB does not terminate, that is, $\alpha_i, \beta_i\neq0$ for $1\leq i \leq k$. Then $B_k$ has full column rank, and thus 
\begin{align*}
	\min_{x=W_ky}\|Ax-b\|_2 
	&= \min_{y\in\mathbb{R}^{k}}\|AW_ky-U_{k+1}\beta_1e_{1}\|_2 \\
	&= \min_{y\in\mathbb{R}^{k}}\|U_{k+1}B_ky-U_{k+1}\beta_1e_{1}\|_2 \\
	&= \min_{y\in\mathbb{R}^{k}}\|B_ky-\beta_1e_{1}\|_2,
\end{align*}
which has a unique solution. Therefore, the solution to \cref{subspace_regu} is
\begin{equation}\label{proj_y}
	x_k = W_ky_k, \ \ \ y_k=\argmin_{y\in\mathbb{R}^{k}}\|B_ky-\beta_1e_{1}\|_2
\end{equation}
for $k=1,2\dots$. This is a very similar procedure to the LSQR solver for least squares problems. In practical computations, there is a recursive formula to update $x_{k+1}$ from $x_k$ without solving the projected problem $\min_{y}\|B_ky-\beta_1e_{1}\|_2$ at each iteration. This will be discussed lated. Now we investigate the structure of solution subspace $\mathcal{S}_k$.

\begin{Prop}\label{prop3.2}
	Let the solution subspace be $\mathcal{S}_k=\mathrm{span}\{W_k\}$. Then 
	\begin{equation}\label{Sk}
		\mathcal{S}_k = \mathrm{span}\{Z(D_{\alpha}^{-1}D_{A}^{T}D_{A})^{i}D_{\alpha}^{-1}D_{A}^{T}U_{A}^{T}b\}_{i=0}^{k-1}
	\end{equation}
\end{Prop}
\begin{proof}
	By \Cref{prop3.1} we have $\mathcal{S}_k = \mathrm{span}\{(G^{-1}A^TA)^iG^{-1}A^Tb\}_{i=0}^{k-1}$.
	Using the GSVD expression \cref{gsvd} and \Cref{prop2.1}, we have
	\begin{align*}
		& G^{-1}A^Tb = ZD_{\alpha}^{-1}Z^{T}Z^{-T}D_{A}^{T}U_{A}^{T}b
		= ZD_{\alpha}^{-1}D_{A}^{T}U_{A}^{T}b , \\ 
		& G^{-1}A^TA = ZD_{\alpha}^{-1}Z^{T}Z^{-T}D_{A}^{T}D_{A}Z^{-1}
		= ZD_{\alpha}^{-1}D_{A}^{T}D_{A}Z^{-1} .
	\end{align*}
	Therefore, we obtain
	\begin{align*}
		(G^{-1}A^TA )^iG^{-1}A^Tb
		&= (ZD_{\alpha}^{-1}D_{A}^{T}D_{A}Z^{-1})^{i}ZD_{\alpha}^{-1}D_{A}^{T}U_{A}^{T}b \\
		&= Z(D_{\alpha}^{-1}D_{A}^{T}D_{A})^{i}Z^{-1}ZD_{\alpha}^{-1}D_{A}^{T}U_{A}^{T}b \\
		&= Z(D_{\alpha}^{-1}D_{A}^{T}D_{A})^{i}D_{\alpha}^{-1}D_{A}^{T}U_{A}^{T}b
	\end{align*}
	for $i=0,1,\dots,k-1$, which is the desired result.
\end{proof}

Note from \cref{Sk} that the solution subspace $\mathcal{S}_k$ is directly related to the TGSVD solution subspace $\mathrm{span}\{Z_k\}$. Hence we can expect that the iterative regularized solution can include prior properties about the desired solution encoded by the regularizer $x^TMx$. We emphasis that the good property of $\mathcal{S}_k$ attributes to the proper choice of preconditioner $R$ used to construct $W_k$. Here the term ``preconditioned" is not aimed at accelerating convergence of iterative solver but rather at forcing some specific regularity into the associated solution subspace. In fact, the regularized solution $x_k$ has a filtered GSVD expansion form, which sheds light on the regularization effect of the proposed algorithm. We will analyze it in the next section.

Now we discuss how to efficiently update solutions and stop the iteration early to get a good final regularized solution. By applying Givens QR factorization to \cref{proj_y}, the updating procedure for $x_k$ can be derived similarly to that for LSQR; see  \cite{Paige1982} for details. Starting from $x_0=\mathbf{0}$, $x_i$ is computed recursively by the following algorithm. 

\begin{algorithm}[htb]
	\caption{Updating procedure}\label{alg2}
	\begin{algorithmic}[1]
		\State {Let $x_{0}=\mathbf{0},\ p_{1}=w_{1}, \ \bar{\phi}_{1}=\beta_{1},\ \bar{\rho}_{1}=\alpha_{1}$} 
		\For{$i=1,2,\ldots,k,$}
		\State $\rho_{i}=(\bar{\rho}_{i}^{2}+\beta_{i+1}^{2})^{1/2}$
		\State $c_{i}=\bar{\rho}_{i}/\rho_{i},\ s_{i}=\beta_{i+1}/\rho_{i}$
		\State$\theta_{i+1}=s_{i}\alpha_{i+1},\ \bar{\rho}_{i+1}=-c_{i}\alpha_{i+1}$
		\State $\phi_{i}=c_{i}\bar{\phi}_{i},\ \bar{\phi}_{i+1}=s_{i}\bar{\phi}_{i} $
		\State $x_{i}=x_{i-1}+(\phi_{i}/\rho_{i})p_{i} $
		\State $p_{i+1}=w_{i+1}-(\theta_{i+1}/\rho_{i})p_{i}$
		\EndFor
	\end{algorithmic}
\end{algorithm}

The identity
\begin{equation}\label{res}
	\bar{\phi}_{k+1}=\|B_ky_{k}-\beta_1 e_{1}\|_2 = \|Ax_k-b\|_2,
\end{equation}
holds, similar to that for LSQR \cite{Paige1982}. Thus the residual norm can be updated very quickly by \Cref{alg2} without explicitly computing $\|Ax_k-b\|_2$. Note that $\bar{\phi}_{k+1}$ decreases monotonically since $x_k$ minimizes $\|Ax-b\|_2$ in the gradually expanding subspace $\mathrm{span}\{W_k\}$.

In order to estimate the optimal early stopping iteration, if we have an accurate estimate of $\|e\|$, the discrepancy principle (DP) is a common choice. It states that we should stop iteration at the first $k$ satisfying 
\begin{equation}\label{discrepancy}
	\bar{\phi}_{k+1}=\|Ax_k-b\|_2 \leq \tau\|e\|_2,
\end{equation}
where $\tau>1$ slightly. The discrepancy principle method usually suffers from under-estimate the optimal $k$ and thus the solution is over-regularized. Another approach is the L-curve criterion, which does not need $\|e\|$ in advance. We plot the L-curve 
\begin{equation}
	\left(\log\|Ax_{k}-b\|_2,\log(x_{k}^TMx_k)^{1/2}\right) = \left(\log\bar{\phi}_{k+1},\log(x_{k}^TMx_k)^{1/2}\right)
\end{equation}
and choose the $k$ corresponding to the corner of it as a good estimate of the optimal early stopping iteration. The whole process of the pGKB based subspace projection regularization (pGKB\_SPR) algorithm is summarized in \Cref{alg3}.

\begin{algorithm}[htb]
	\caption{pGKB based subspace projection regularization (pGKB\_SPR)}\label{alg3}
	\begin{algorithmic}[1]
		\Require $A$, $M$, $\alpha>0$, $b$, $x_{0}=\mathbf{0}$  \Comment{require $\tau\|e\|_2$ for DP}
		\For{$k=1,2,\ldots,$}
		\State Compute $u_k$, $w_k$, $\alpha_k$, $\beta_k$ by \texttt{pGKB} 
		\State Compute $\rho_k$, $\theta_{k+1}$, $\bar{\rho}_{k+1}$, $\phi_{k}$, $\bar{\phi}_{k+1}$ by \texttt{updating procedure}
		\State Compute $x_k$, $p_{k+1}$ by \texttt{updating procedure} 
		\If{Early stopping criterion is satisfied}  \Comment{DP or L-curve}
		\State Denote the estimated iteration by $k_1$, terminate iteration at $k_1$
		\EndIf
		\EndFor
		\Ensure Final regularized solution $x_{k_1}$  
	\end{algorithmic}
\end{algorithm}

At the end of this section, we discuss several implementation details for the pGKB process, which are important for increasing computational efficiency, especially for large-scale problems. There are three main issues that need to be considered.
\begin{itemize}
	\item The pGKB process has the structure of nested inner-out iteration, where at each outer iteration, a linear system $Gs=A^Tu_i$ needs to be solved. Although for some cases that $A$ and $M$ have special structures, this system can be solved via a direct matrix factorization, for most large-scale problems, however, the only advisable approach is using an iterative solver, such as the conjugate gradient (CG) or minimum residual (MINRES) algorithm. Meanwhile, The value of $\alpha>0$ should be set such that the condition number of $G=A^TA+\alpha M$ is small to make the iterative solver converge  quickly. The default value $\alpha=1$ is often fruitful, and we also try other moderate values $\alpha\in [0.001, 100]$ in the numerical experiments.
	\item For large-scale problems, iteratively solving $Gs=A^Tu_i$ may still be costly, especially when the solution accuracy is high. Since there is a limit of the accuracy of the best regularized solution, which is affected by the noise level and other factors, we can expect that the accuracy of the final regularized solution will not be affected even if inner iterations are not computed exactly. When the noise level $\varepsilon=\|e\|_2/\|b_{\text{true}}\|_2$ is not extremely small, we numerically find that the solution accuracy of $Gs=A^Tu_i$ can be relaxed considerably (the stopping tolerance for the MATLAB function \texttt{pcg.m} can be set larger than the default value \texttt{1e-6}), which improves the overall efficiency of the algorithm.
	\item Due to the computational errors and rounding errors coming from inaccurate inner iterations and finite precision arithmetic, respectively, the 2- and $G$-orthogonality of computed $u_{i}$ and $w_i$ gradually lose, which leads to a delay of convergence of iterative solutions. In our implementation, we do full reorthogonalization on $u_i$ and $w_i$ to maintain their numerical 2- and $G$-orthogonality to ensure the normal convergence behavior.
\end{itemize}
In this paper, the above issues are only investigated numerically. We will make theoretical analysis of the required accuracy of inner iterations as well as efficient reorthogonalization strategies in the forthcoming work.

\section{Regularization effect of pGKB\_SPR}\label{sec4}
The dominant GSVD components of $\{A,L\}$ can be approximated well by the solution subspaces generated by pGKB as the iteration proceeds. This is a desirable property for a regularization method to ensure that the obtained solution captures main information corresponding to dominant right generalized singular vectors while filtered out noise corresponding to others.

\begin{Lem}\label{lanc_process}
	Suppose the singular values of $B_k$ are $\theta_{1}^{(k)}>\theta_{2}^{(k)}>\cdots>\theta_{k}^{(k)}>0$. Then $(\theta_{i}^{(k)})^2$ converges to one of the generalized eigenvalues of 
	$A^TAz=\xi Gz$ as $k$ increases.
\end{Lem}
\begin{proof}
	Since $\alpha_1w_1=G^{-1}A^Tu_1$ and $\beta_1u_1=b$, we have 
	\begin{equation}\label{lanc1}
		\alpha_1\beta_1w_1 = G^{-1}A^Tb .
	\end{equation}
	By \cref{GKB23,GKB33}, we have
	\begin{align}\label{lanc2}
		A^TAW_k
		&= A^TU_{k+1}B_k  \nonumber \\
		&= (GW_kB_{k}^T+\alpha_{k+1}Gw_{k+1}e_{k+1}^T)B_k \nonumber \\
		&= GW_k(B_{k}^{T}B_k)+\alpha_{k+1}\beta_{k+1}Gw_{k+1}e_{k+1}^T .
	\end{align}
	Define $T_k$ as
	\[T_k= B_{k}^{T}B_k=
	\begin{pmatrix}
		\alpha_{1}^{2}+\beta_{2}^{2} & \alpha_2\beta_2 & & \\
		\alpha_2\beta_2 & \alpha_{2}^{2}+\beta_{3}^{2} & \ddots & \\
		 & \ddots & \ddots & \alpha_k\beta_k \\
		  & & \alpha_k\beta_k & \alpha_{k}^{2}+\beta_{k+1}^{2}
	\end{pmatrix} ,
	\]
	which is a symmetric tridiagonal matrix. Combining \cref{lanc1,lanc2} we find that $T_k$ is the Ritz-Galerkin projection of $A^TA$ onto subspace $\mathrm{span}\{W_k\}$ that is generated by the Lanczos tridiagonalization process of $A^TA$ with starting vector $G^{-1}A^Tb$ under the $G$-inner product \cite[\S 5.5]{Bai2000}.
	
	By the convergence theory of Lanczos tridiagonalization for generalized eigenvalue problem of matrix pair $\{A^TA, G\}$, the eigenvalues of $T_k$ will converge to the generalized eigenvalues of $A^TAz=\xi Gz$ as $k$ increases; see \cite{Hetmaniuk2006} and \cite[\S 10.1.5]{Golub2013}. Note that the $i$-th eigenvalue of $T_k$ is $(\theta_{i}^{(k)})^2$. The proof is completed.
\end{proof}

The convergence behavior of $\theta_{i}^{(k)}$ is governed by the same Kaniel-Paige-Saad theory as in the standard case for a single matrix; see e.g.\cite[\S 10.1.5]{Golub2013}. It states that we get a faster convergence to those generalized eigenvalues that are in the two ends of the spectrum and well separated. For simplicity of expression, we use $\xi_i$ to denote the $i$-th largest generalized eigenvalue of $A^TAz=\xi Gz$. By \cref{prop2.1}, it means $\xi_1=\cdots=\xi_r=1$, $\xi_i=\gamma_{i}^2/(\gamma_{i}^2+\alpha)$ for $r+1\leq i\leq r+q$ and $\xi_{r+q+1}=\cdots=\xi_n=0$. Note that $\xi_i$ are decreasing and clustered at zero. Thus we can get a faster convergence of $(\theta_{i}^{(k)})^2$ to some largest generalized eigenvalues $\xi_i$, while the corresponding generalized eigenvectors $z_i$ are also preferentially approximated by Ritz vectors extracted from $\mathrm{span}\{W_k\}$. Since the regularized solution lies in $\mathrm{span}\{W_k\}$, this ensures that the dominant information encoded by some leading $z_i$ can be captured by pGKB\_SPR while the noisy components are filtered out. To be more precisely, we have the following result.

\begin{Thm}\label{filter_express}
	Suppose $A$ has full column rank, which means $r+q=n$ in the GSVD of $\{A,L\}$. Let $\sigma_i=1$ for $1\leq i\leq r$. Then the $k$-th regularized solution obtained by pGKB\_SPR can be written as
	\begin{equation}\label{filter_solution}
		x_k = \sum_{i=1}^{n}f_{i}^{(k)}\frac{u_{A,i}^{T}b}{\sigma_i}z_i
	\end{equation}
	with filter factors
	\begin{equation}\label{filter1}
		f_{i}^{(k)} = 1 - \prod_{j=1}^{k}\frac{(\theta_{j}^{(k)})^2-\xi_i}{(\theta_{j}^{(k)})^2}, \ \ i=1,\dots,n.
	\end{equation}
\end{Thm}
\begin{proof}
	Note from \cref{gsvd} that $\mathrm{rank}(A)=r+q$, which implies $r+q=n$ if $A$ has full column rank. In this case $\sigma_i>0$ for all $1\leq i\leq n$. Following the second proof of \Cref{prop3.1}, $x_k$ is the solution of 
	\begin{align*}
		\min\limits_{x\in\mathrm{span}\{W_k\}}\|Ax-b\|_2
		= \min\limits_{Rx\in R\mathrm{span}\{W_k\}}\|AR^{-1}Rx-b\|_2 
		= \min\limits_{\bar{x}\in\mathrm{span}\{V_k\} \atop x=R^{-1}\bar{x}}\|\bar{A}\bar{x}-b\|_2,
	\end{align*}
	which means that $x_k=R^{-1}\bar{x}_k$ with $\bar{x}_k=\argmin_{\bar{x}\in\mathrm{span}\{V_k\}}\|\bar{A}\bar{x}-b\|_2$. Since $\{v_i\}_{i=1}^{k}$ are the right Lanczos vectors generated by the standard GKB process of $\bar{A}$ with starting $b$, we know $\bar{x}_k$ is the $k$-th CGLS solution applied to the normal equation $\bar{A}^T\bar{A}=\bar{A}^Tb$, where the initial solution is set as $\boldsymbol{0}$ \cite[\S 6.3]{Hansen1998}. Since $\bar{A}=AR^{-1}$ has full column rank, using \cite[Property 2.8]{Van1986rate}, the expression of $\bar{x}_k$ is given by
	\begin{equation}\label{filter2}
		\bar{x}_k = (I-\mathcal{R}_k(\bar{A}^T\bar{A}))(\bar{A}^T\bar{A})^{-1}\bar{A}^Tb ,
	\end{equation}
	where $\mathcal{R}_k$ is the so-called Ritz polynomial 
	\begin{equation*}
		\mathcal{R}_k(\theta) = \prod_{j=1}^{k}\frac{(\theta_{j}^{(k)})^2-\theta}{(\theta_{j}^{(k)})^2} .
	\end{equation*}
	From the proof of \Cref{prop2.1} we have $Z^TGZ=D_{\alpha}$. Let $\bar{Z}=ZD_{\alpha}^{-1/2}$. Then $\bar{Z}$ is a $G$-orthonormal matrix, and $\bar{A}=U_AD_AZ^{-1}R^{-1}=U_A(D_AD_{\alpha}^{-1/2})(R\bar{Z})^{-1}$. Therefore, we have $\bar{A}^T\bar{A}=(R\bar{Z})^{-T}(D_{A}^TD_AD_{\alpha}^{-1})(R\bar{Z})^{-1}$. Using the fact
	\[(R\bar{Z})^{-1}(R\bar{Z})^{-T} = [(R\bar{Z})^{T}(R\bar{Z})]^{-1} = (\bar{Z}^TG\bar{Z})^{-1}=I , \]
	we obtain
	\begin{align*}
		I-\mathcal{R}_k(\bar{A}^T\bar{A})
		= (R\bar{Z})^{-T}(I-\mathcal{R}_k(D_{A}^TD_AD_{\alpha}^{-1}))(R\bar{Z})^{-1} 
		= (R\bar{Z})^{-T}\Lambda(R\bar{Z})^{-1} ,
	\end{align*}
	where $\Lambda = \mathrm{diag}(\{1-\mathcal{R}_k(\xi_i)\}_{i=1}^{n})$. Substituting the expression of $I-\mathcal{R}_k(\bar{A}^T\bar{A})$ into \cref{filter2}, we get
	\begin{align*}
		\bar{x}_k &= (R\bar{Z})^{-T}\Lambda D_{\alpha}^{1/2}(D_{A}^{T}D_A)^{-1}D_{A}^{T}U_{A}^{T}b \\
		&= R\bar{Z}\Lambda D_{\alpha}^{1/2}(D_{A}^{T}D_A)^{-1}D_{A}^{T}U_{A}^{T}b \\
		&= RZD_{\alpha}^{-1/2}\Lambda D_{\alpha}^{1/2}(D_{A}^{T}D_A)^{-1}D_{A}^{T}U_{A}^{T}b \\
		&= RZ\Lambda (D_{A}^{T}D_A)^{-1}D_{A}^{T}U_{A}^{T}b .
	\end{align*}
	Therefore, we finally obtain
	\begin{equation*}
		x_k = R^{-1}\bar{x}_k = \sum_{i=1}^{n}(1-\mathcal{R}_k(\xi_i))\frac{u_{A,i}^{T}b}{\sigma_i}z_i,
	\end{equation*}
	and the filter factors are $f_{i}^{(k)}=1-\mathcal{R}_k(\xi_i)$ for $i=1,\dots,k$.
\end{proof}

This theorem shows that $x_k$ has a filtered GSVD expansion from. If the $k$ Ritz values $\{(\theta_{j}^{(k)})^2\}_{j=1}^{k}$ approximate the first $k$ generalized singular values $\{\xi_i\}_{i=1}^{k}$ in natural order, i.e. $(\theta_{j}^{(k)})^2\approx\xi_i$ for $i=1,\dots,k$, from \cref{filter1} we can justify that $f_{i}^{(k)}\approx 1$ for $i=1,\dots,k$ and $f_{i}^{(k)}$ decreases monotonically to zero for $i=k+1,\dots,n$. This means that $x_k$ mainly contains the first $k$ dominant GSVD components and filters out the others. Therefore, the pGKB\_SPR algorithm exhibits the typical semi-convergence behavior, which means that $\|x_k-x_{\text{true}}\|_2$ first decreases, then as $k$ becomes fairly large, $x_k$ will diverge from $x_{\text{true}}$ since too many noisy components are included. The transition point is called the semi-convergence point, at which we get the best regularized solution.

Finally, we give a connection of the pGKB based regularization method with another common regularization method, which is based on the joint bidiagonalization of $\{A,L\}$. 

\begin{Thm}
	If $\alpha=1$, the solution subspaces generated by the pGKB process of $\{A, M\}$ are the same as those generated by the joint bidiagonalization of $\{A, L\}$.
\end{Thm}
\begin{proof}
	Note that $D_{\alpha}=I$ when $\alpha=1$. By \Cref{prop3.2}, we have
	\[\mathcal{S}_k = \mathrm{span}\{Z(D_{A}^{T}D_{A})^{i}D_{A}^{T}U_{A}^{T}b\}_{i=0}^{k-1} .\]
	By checking the solution subspace given in \cite{Kilmer2007}, we find that $\mathcal{S}_k$ is just the subspace generated by the joint bidiagonalization of $\{A, L\}$.
\end{proof}

The joint bidiagonalization based regularization method is effective when $L$ is available. If $L$ is not known or the square root decomposition of $M$ is costly, the pGKB based regularization method is a very good alternative.

\section{Hybrid regularization method based on pGKB}\label{sec5}
Although the DP or L-curve criterion can be used to stop iteration early to avoid semi-convergence, the corresponding solution is  still often over or under-regularized since the relative error is very sensitive near the semi-convergence point. The hybrid regularization method is another type of iterative method that can stabilize the convergence behavior, which usually applies Tikhonov regularization to the projected problem at each iteration; see e.g. \cite{Kilmer2001,Chungnagy2008,Renaut2017}.

Using subspace projection procedure with $\mathcal{S}_k=\mathrm{span}\{W_k\}$, the Tikhonov regularization problem becomes 
\begin{align*}
	\min_{x=W_ky}\{\|Ax-b\|_{2}^2+\lambda x^TMx\} 
	= \min_{y\in\mathbb{R}^{k}}\{\|B_ky-\beta_1e_1\|_{2}^2+\lambda y^T(W_{k}^{T}MW_{k})y\} .
\end{align*}
Note that $W_{k}^{T}MW_{k}\in\mathbb{R}^{k\times k}$. First we compute the square root decomposition $W_{k}^{T}MW_{k}=C_{k}^{T}C_{k}$, which can be done directly by the eigenvalue decomposition of $W_{k}^{T}WZ_{k}$ for small $k$. The hybrid method solves the regularized projected problem
\begin{align}\label{hyb}
	y_{k}^{\mu_k} = \argmin_{y\in\mathbb{R}^{k}}\{\|B_k y-\beta_1e_1\|^2+\mu_k\|C_k y\|^2\},
\end{align}
and let $x_{k}^{\mu_k}=W_ky_{k}^{\mu_k}$ for $k=1,2,\dots$, where $\mu_k$ should be determined at each iteration. Let $\lambda^{\text{opt}}$ and $\mu_{k}^{\text{opt}}$ denote the optimal regularization parameters of the original problem and $k$-th projected problem, respectively. The main idea of the hybrid method is that as $k$ gradually increases, the projected problem approximates the original problem. Then for a large enough $k$, we can hope that $\mu_{k}^{\text{opt}}$ converges to $\lambda^{\text{opt}}$ and thus the corresponding regularized solution will also converge \cite{Kilmer2001,Renaut2017}.

In order to determine $\mu_{k}^{\text{opt}}$ at each iteration, we adapt the weighted generalized cross validation method (WGCV), which was first proposed in \cite{Chungnagy2008} for standard-form Tikhonov regularization. From \cref{hyb} we have $y_{k}^{\mu}=(B_{k}^{T}B_{k}+\mu C_{k}^{T}C_{k})^{-1}B_{k}^{T}\beta_{1}e_{1}:=B_{k,\mu}^{\dag}\beta_{1}e_{1}$. At the $k$-th step, the WGCV method finds the minimizer of the following function with the weight parameter $\omega_k$:
\begin{equation}\label{wgcv}
	G(\omega_k,\mu) = \frac{\|B_{k}y_{k}^{\mu}-\beta_{1}e_{1}\|_{2}^{2}}{\left(\mathrm{trace}(I-\omega_{k}B_{k}B_{k,\mu}^{\dag})\right)^2} ,
\end{equation}
and use this minimizer as $\mu_k$. We remark that if $w_k=1$ for all $k$, it is the standard GCV method. The weight parameter $\omega_k$ is initialized and automatically updated following the same strategy in \cite{Chungnagy2008}. By writing the analytical expression of $G(\omega_k,\mu)$ using the GSVD of $\{B_k, C_k\}$, we can find its minimizer using the MATLAB build-in function \texttt{fminbnd.m}. This approach is at cost of $O(k^3)$ flops, dominated by computing the GSVD of $\{B_k, C_k\}$. In the ideal situation, as $k$ increases, $\mu_k$ will converge and $G(1,\mu_k)$ will converge to a fixed value. We terminate the iteration at $k+s_{1}$ with $k$ the first iteration satisfying
\begin{equation}\label{tol1}
\Big| \frac{G(1,\mu_{i+1})-G(1,\mu_{i})}{G(1,\mu_{1})} \Big| < \mathtt{tol1}, \ \
i = k,\dots, k+s_{1} ,
\end{equation}
where $s_1+1$ is the length of window to avoid bumps. We set $s_1=4$ and $\mathtt{tol1}=10^{-6}$ by default.

If we have a good estimate of $\|e\|_2$, there is another method that can update $\mu_k$ step by step quickly based on DP. This method is called ``secant update" (SU), which was first proposed for the Arnoldi-Tikhonov hybrid method \cite{Gazzola2014automatic}. Here we show how this method is adapted to update $\mu_k$ for the pGKB based hybrid method. A heuristic derivation is as follows. At each iteration we define the function
\begin{equation}\label{psi}
	\psi_{k}(\mu)=\|Ax_{k}^{\mu}-b\|_2=\|B_{k}y_{k}^{\mu}-\beta_{1}e_{1}\|_2 .
\end{equation}
In order to estimate $\mu_k$ by DP, we consider to determine the proper $k$ and $\mu_k$ simultaneously by solving the nonlinear equation $\psi_{k}(\mu)=\tau\|e\|_2$ with a fixed $\tau>1$ slightly. We remark that this equation have a solution only when $k$ is sufficiently large. We use the following secant method to solve the above equation.
Starting from an initial value $\mu_0$, suppose we already have $\mu_{k-1}$ at the $(k-1)$-th iteration. Notice that $\psi_{k}(\mu)$ monotonically increases with respect to $\mu$ \cite{lewis2009arnoldi}. It can be approximated by the linear function
\begin{equation*}
f(\mu) = \psi_{k}(0) + \frac{\psi_{k}(\mu_{k-1})-\psi_{k}(0)}{\mu_{k-1}}\mu ,
\end{equation*}
which interpolates $\psi_{k}(\mu)$ at $0$ and $\mu_{k-1}$. To update $\mu_{k}$ for the next step, we replace $\psi_k(\mu)$ by $f(\mu)$ and solve $f(\mu)=\tau\|e\|_2$, which leads to 
\begin{equation*}
\mu_{k}=\frac{\tau\|e\|_2-\psi_{k}(0)}{\psi_{k}(\mu_{k-1})-\psi_{k}(0)}\mu_{k-1} .
\end{equation*}
This update formula may suffer from instability for small $k$, since it holds that $\psi_{k}(0)>\tau\|e\|_2$. Therefore, we finally use 
\begin{equation}\label{5.4}
\mu_{k}=\Big|\frac{\tau\|e\|_2-\psi_{k}(0)}{\psi_{k}(\mu_{k-1})-\psi_{k}(0)}\Big|\mu_{k-1}  
\end{equation}
as the practical formula to update $\mu_k$. We set $\lambda_{0}=1.0$ by default. To avoid extra computations for $\psi_{k}(\mu_k)$, we use $\mu_{k-1}$ as the estimated regularized parameter for the $k$-th iteration. Numerically we find that \eqref{5.4} is very stable in the sense that when $k$ is sufficient large, both the regularization parameter $\mu_{k}$ and residual norm $\psi_{k}(\mu_{k-1})$ tend to plateau.

Since $\psi_{k}(0)=\|Ax_{k}-b\|_2=\bar{\phi}_{k+1}$, which is the residual norm of the $k$-th pGKB\_SPR solution, it can be updated efficiently by \Cref{alg2} (without computing $p_k$ and $x_k$). To terminate the iteration, we choose $k+s_{2}$ as the stopping iteration with $k$ the first iteration satisfying
\begin{equation}\label{tol2}
\psi_{k}(0) \leq \tau\|e\|_2 \ \ \mathrm{and} \ \
\Big|\frac{\psi_{i+1}(\mu_{i})-\psi_{i}(\mu_{i-1})}{\psi_{i}(\mu_{i-1})} \Big|\leq \mathtt{tol2} ,  \ \
i = k,\dots, k+s_{2} ,
\end{equation}
where $s_2+1$ is the length of window to avoid bumps. We set $s_2=4$ and $\mathtt{tol2}=0.001$ by default.

To summarize, we show pseudocode of the pGKB based hybrid regularization (pGKB\_HR) algorithm using WGCV or SU in \Cref{alg4}.
\begin{algorithm}[htb]
	\caption{pGKB based hybrid regularization (pGKB\_HR)}\label{alg4}
	\begin{algorithmic}[1]
		\Require $A$, $b$, $M$, $\alpha>0$, $\omega_1$ or $\mu_0$  \Comment{require $\tau\|e\|_2$ for SU}
		\For{$k=1,2,\ldots,$}
		\State Compute $u_k$, $w_k$, $\alpha_k$, $\beta_k$ by \texttt{pGKB} 
		\State Compute decomposition $W_{k}^{T}MW_{k}=C_{k}^{T}C_{k}$ to get $C_k$
		\If{`hybrid = WGCV'} \Comment{WGCV}
		\State Compute the GSVD of $\{B_k, C_k\}$
		\State Determine $\mu_k$ by minimizing \cref{wgcv}
		\State Compute $y_{k}^{\mu_{k}}$ by solving \cref{hyb} 
		\State Update $\omega_k$ following \cite{Chungnagy2008}
		\ElsIf{`hybrid = SU'}  \Comment{SU}
		\State Compute $y_{k}^{\mu_{k-1}}$ by solving \cref{hyb} 
		\State Compute the residual norm $\psi_{k}(\mu_{k-1})$ by \cref{psi}
		\State Compute $\bar{\phi}_{k+1}$ by \texttt{updating procedure}
		\State Update $\mu_k$ by \cref{5.4} 
		\EndIf
		\State Terminate iteration by \cref{tol1} or \cref{tol2} at $k_2$ 
		\EndFor
		\State Compute $x_{k_2}^{\mu_{k_2}}=W_{k_2}y_{k_2}^{\mu_{k_{2}}}$
		\Ensure Final regularized solution $x_{k_2}$  
	\end{algorithmic}
\end{algorithm}

\section{Experimental results}\label{sec6}
We use some numerical examples to show effectiveness and performance of the proposed algorithms, including pGKB\_SPR with DP and L-curve as early stopping criteria and pGKB\_HR with WGCV and SU for determining $\mu_k$. All experiments are performed with MATLAB R2019b. The codes are available at \url{https://github.com/Machealb/InverProb_IterSolver}, where some codes in the packages of \cite{Hansen2007,Gazzola2019} are exploited. We compare accuracy of the regularized solutions and show convergence behaviors by using the relative reconstruction error
\begin{equation}\label{RE}
RE(k) = \frac{\|x_k-x_{\text{true}}\|}{\|x_{\text{true}}\|} \ \ \  \text{or} \ \ \
RE(k) = \frac{\|x_{k}^{\mu_k}-x_{\text{true}}\|}{\|x_{\text{true}}\|}
\end{equation}
to plot the convergence curve of each algorithm.

\subsection{Small-scale inverse problems}
For the first test problem, we choose {\sf deriv2} from \cite{Hansen2007} by setting $m=n=2000$, which is a discretization of the first kind Fredholm integral equation 
\begin{equation}\label{deriv}
	b(s) = \int_{0}^{1}K(s,t)x(t)\mathrm{d}t, \ \ \ 
	K(s, t) = \begin{cases}
		s(t-1), \ \ s<t , \\
		t(s-1), \ \ s\geq t,
	  \end{cases}
\end{equation}
where $x(t)=t$ and $(s, t) \in [0, 1]^2$. For the second test problem, we use the Gaussian convolution of a 1D piecewise constant signal
\begin{equation}\label{gauss}
	b(s) = \int_{-\infty}^{+\infty}K(s-t)x(t)\mathrm{d}t, \ \ \ 
	K(s-t) = \frac{1}{\sqrt{2\pi}\sigma}\exp\left( -\frac{(s-t)^2}{2\sigma^2} \right),
\end{equation}
where $K$ is the Gaussian kernel with $\sigma=10$, and the discretized signal $x(t)$ on 800 uniform grids over $[0,1]$ is shown in \Cref{fig1}. The Gaussian kernel $K$ is discretized correspondingly with zero boundary condition such that $A\in\mathbb{R}^{800\times 800}$ is a Toeplitz matrix. This problem is named as {\sf gauss1d}. For each problem, we add a Gaussian white noise with noise level $\varepsilon = \|e\|_2/\|b_{\text{true}}\|_2$ to $b_\text{true}$ and form $b=b_{\mathrm{true}}+e$, where we set $\varepsilon=5\times 10^{-4}$ and $\varepsilon=5\times 10^{-3}$ for {\sf deriv2} and {\sf gauss1d}, respectively. The true solutions and noisy observed data for these two problems are shown in \Cref{fig1}.

\begin{figure}[htbp]
	\centering
	\subfloat 
	{\label{fig:1a}\includegraphics[width=0.35\textwidth]{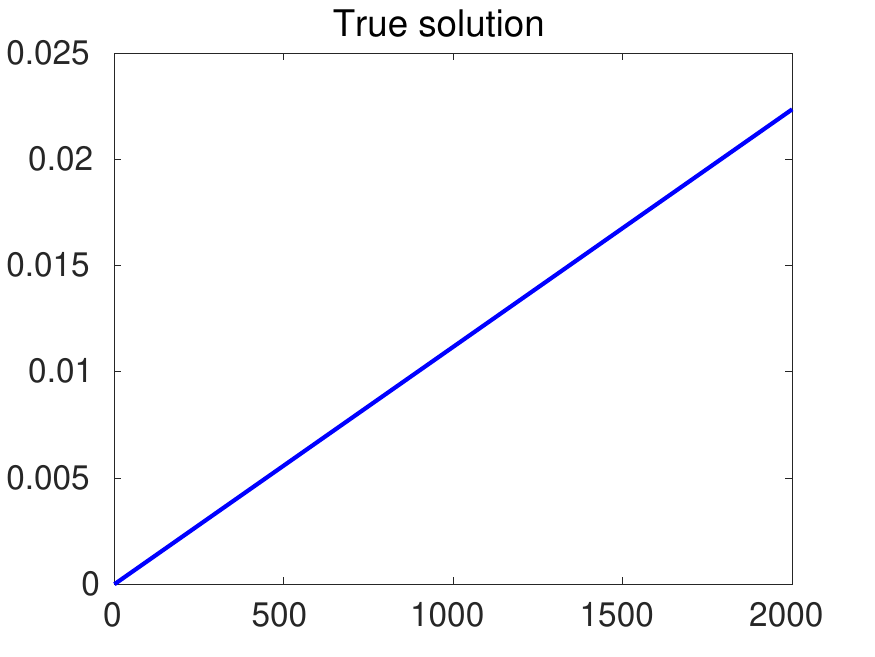}}
	\subfloat{\label{fig:1b}\includegraphics[width=0.35\textwidth]{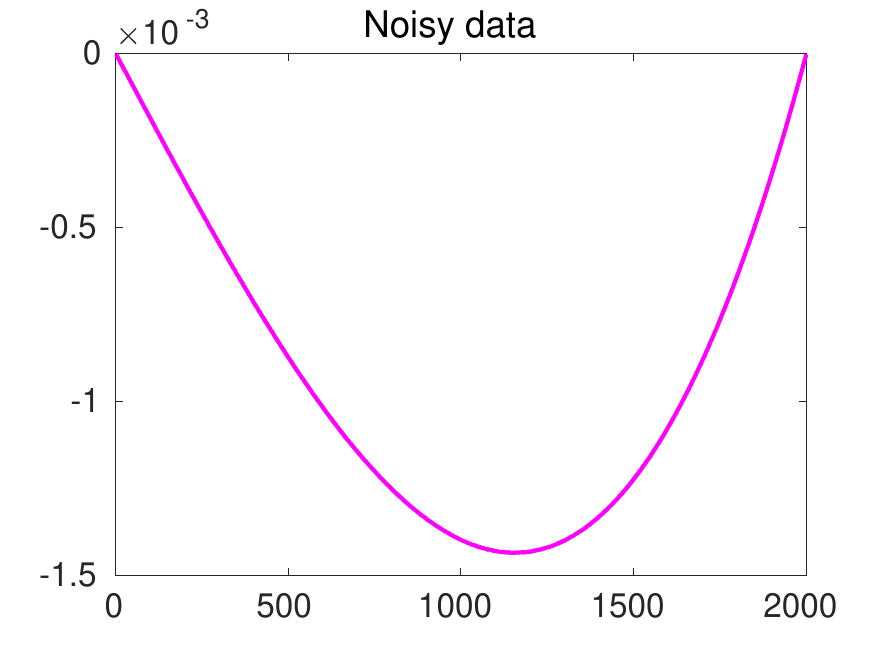}}
	\vspace{-3mm}
	\subfloat 
	{\label{fig:1c}\includegraphics[width=0.35\textwidth]{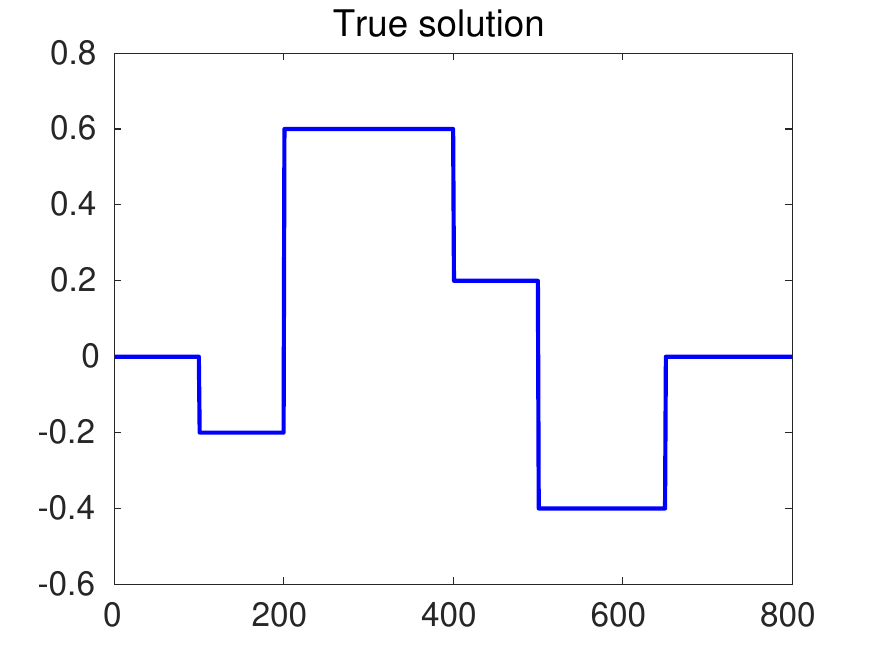}}
	\subfloat{\label{fig:1d}\includegraphics[width=0.35\textwidth]{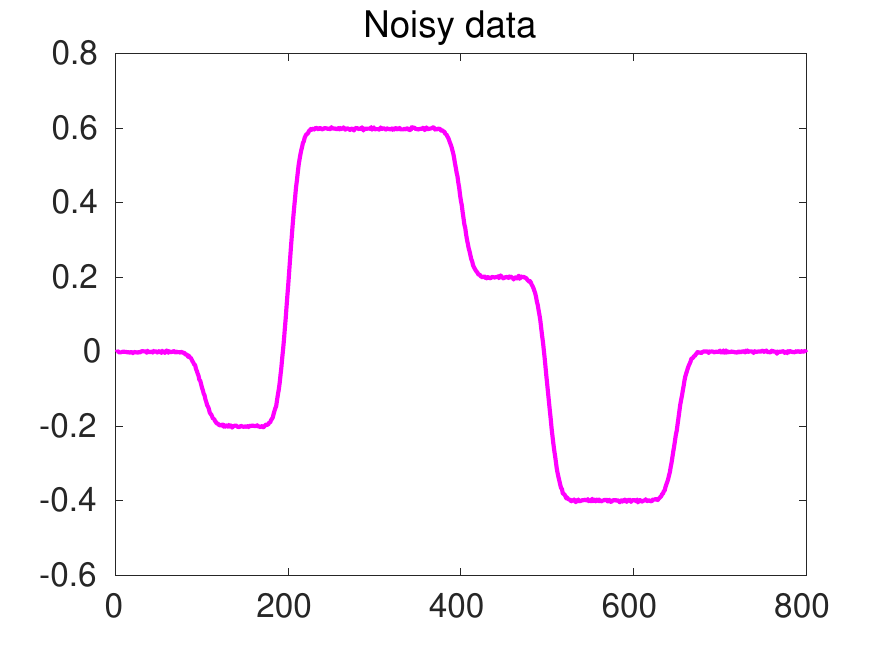}}
	\caption{Illustration of the true solution and noisy observed data. Top: true and observed functions for {\sf deriv2}. Bottom: true and noisy deblurred signals for {\sf gauss1d}.}
	\label{fig1}
\end{figure}

For {\sf deriv2}, we set the regularization matrix $L\in\mathbb{R}^{(n-1)\times n}$ as the scaled discretization of 1D differential operator, which is a bidiagonal matrix with one more row than columns and values $-1$ and $1$ on the subdiagonal and diagonal parts, respectively. Then we form $M=L^TL$ to get $M$. Since $L$ is available, we can compare the pGKB\_SPR with JBD based subspace regularization (JBD\_SPR) algorithm.

\begin{figure}[htbp]
	\centering
	\subfloat[$\tau=0$]
	{\label{fig:2a}\includegraphics[width=0.42\textwidth]{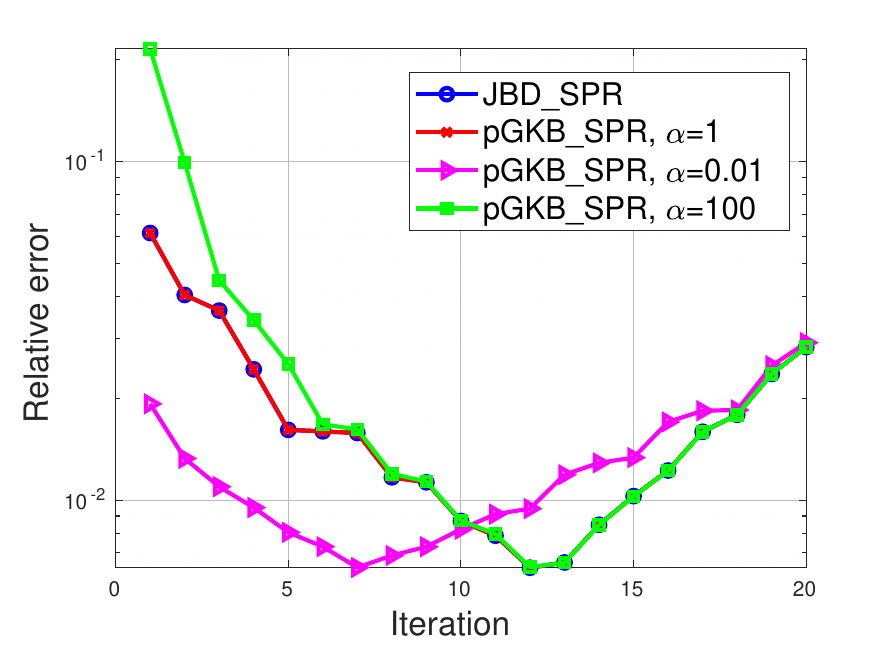}}
	\subfloat[$\alpha=10$]
	{\label{fig:2b}\includegraphics[width=0.42\textwidth]{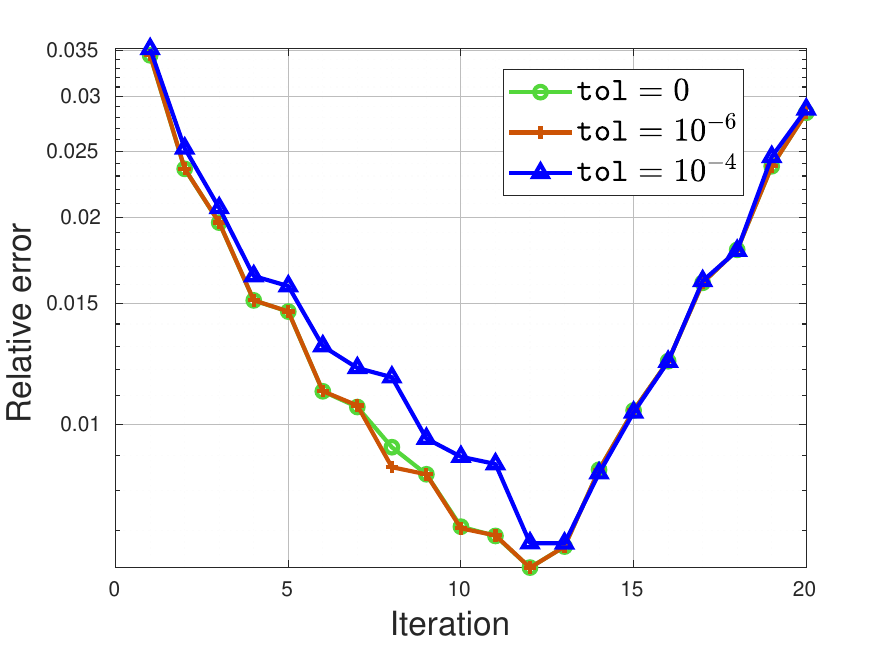}}
	\caption{Comparisons of semi-convergence curves between JBD and pGKB based subspace projection regularization algorithms for {\sf deriv2}. (a) Comparison for different $\alpha$ where inner iterations are computed accurately. (b) Comparison for different solution accuracy of inner iterations.}
	\label{fig2}
\end{figure}

First we illustrate the regularization effect of pGKB\_SPR by plotting its semi-convergence curve and compare it with JBD\_SPR. In this experiment, the inner iteration is computed accurately using matrix inversion, and the values of $\alpha$ are set as $1$, $0.001$ and $100$. From \Cref{fig:2a}, we find that pGKB\_SPR exhibit typical semi-convergence behavior, and the relative errors at the semi-convergence point for the three different $\alpha$ are the same as that of JBD\_SPR. This confirm that the pGKB based regularization method has good regularization effect. To illustrate the impact of computing accuracy of inner iteration on the accuracy of regularized solution, for pGKB\_SPR with $\alpha=10$, we use the MATLAB function \texttt{pcg.m} to solve $Gs=A^Tu_i$ at each outer iteration with stopping tolerance $\texttt{tol}=10^{-6}, 10^{-4}$, and plot the semi-convergence curves in \cref{fig:2b}. We remark that $\texttt{tol}=0$ means that the inner iteration computed accurately using matrix inversion. We find that for $\texttt{tol}=10^{-6}$ the best regularized solution have the same relative error at that for $\texttt{tol}=0$, and the two curves almost coincide for many steps even after semi-convergence. For a more lower computing accuracy with $\texttt{tol}=10^{-4}$, there is a loss of accuracy for the regularized solution. The required solution accuracy of inner iteration impacts on the accuracy of the algorithm, but a theoretical analysis of it is complicated. This will be a part of our forthcoming work.

\begin{figure}
	\centering
	\subfloat[Convergence history]{
		\begin{minipage}[b]{.5\linewidth}
			\includegraphics[width=2.5in]{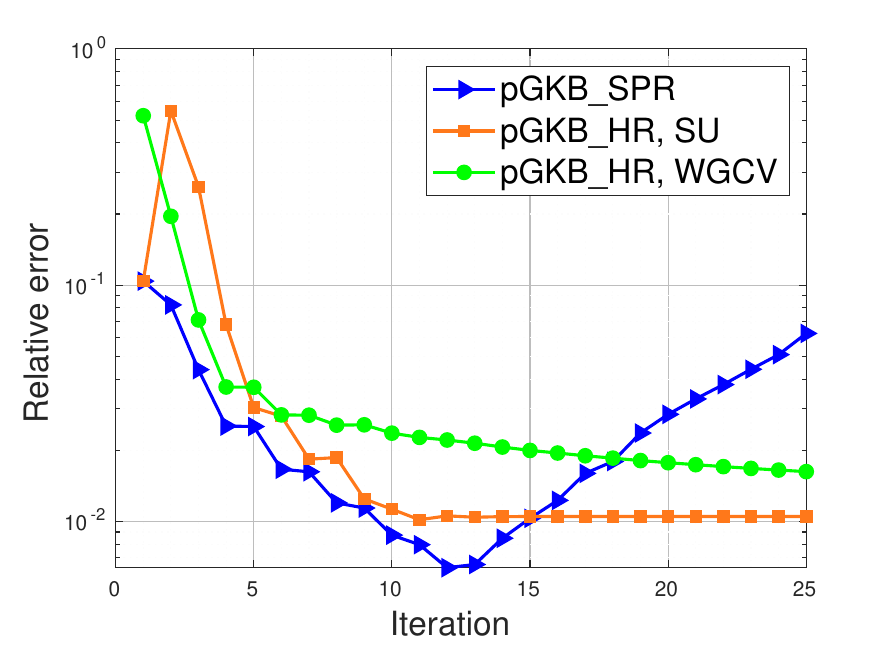}
		\end{minipage}
	} 
	\hspace{-10mm}
	\subfloat[Reconstructed solution]{
		\begin{minipage}[b]{.5\linewidth}
			\centering
			\includegraphics[width=1.33in]{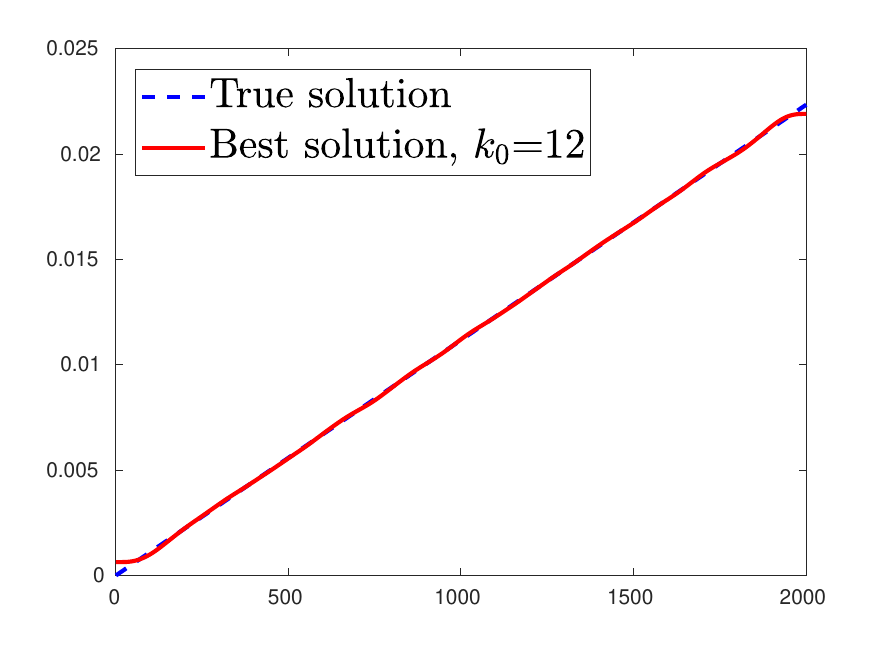}\hspace{-4mm}
			\includegraphics[width=1.32in]{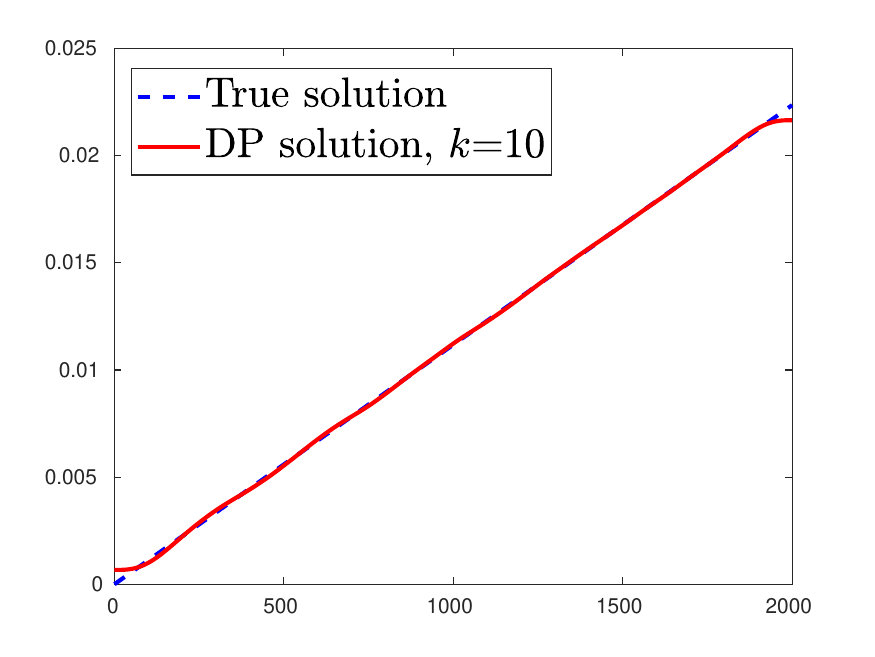}
			\vspace{-1.8mm}
			\includegraphics[width=1.32in]{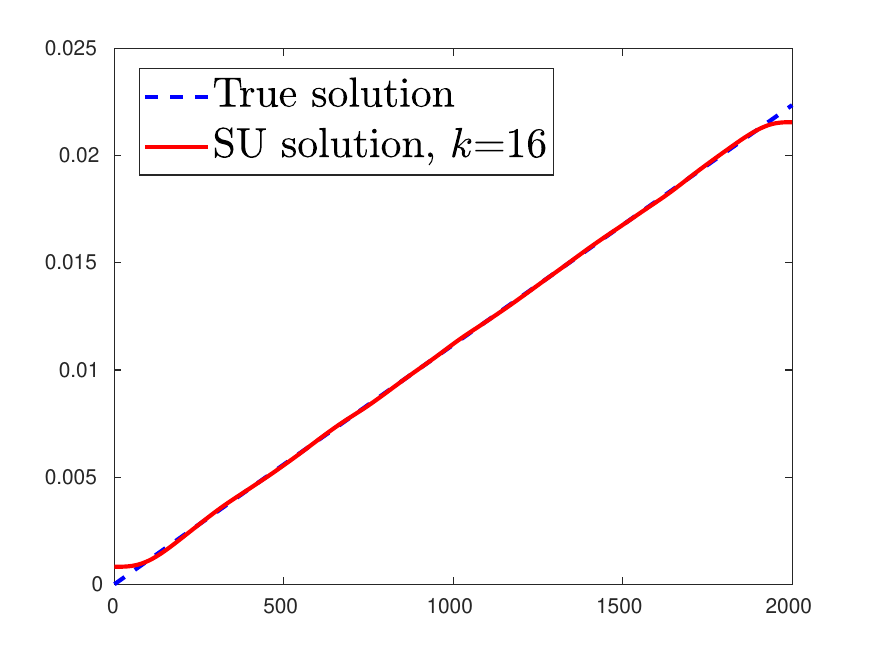}\hspace{-4mm}
			\includegraphics[width=1.33in]{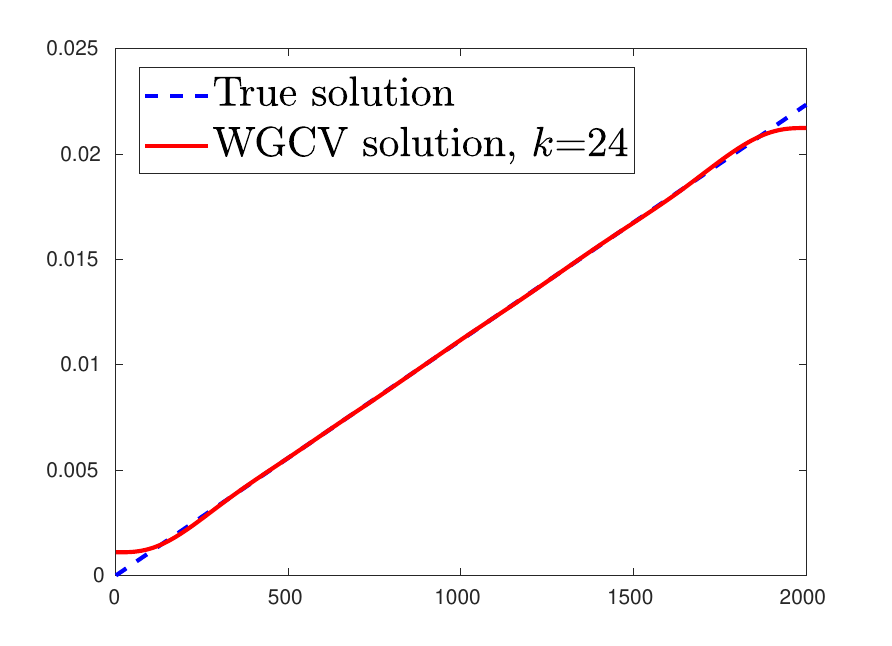}
		\end{minipage}
	}
	\caption{The relative error curves of regularized solutions computed by the pGKB based algorithms and corresponding reconstructed solutions for {\sf deriv2}.}
	\label{fig3}
\end{figure}

Then we illustrate the convergence behavior of pGKB\_SPR and pGKB\_HR, where DP, L-curve (LC) are used to stop iteration early and WGCV, SU are used to determine $\mu_k$. In this experiment, we set $\alpha=10$ and $\texttt{tol}=10^{-6}$. The convergence history curves and the corresponding reconstruct solutions are shown in \Cref{fig3}. Note that the reconstructed solution for pGKB\_SPR with LC are omitted since it is similar as that for pGKB\_SPR with DP. The relative errors of the final regularized solutions and the corresponding iteration number are shown in \Cref{tab1}. Although both DP and LC under-estimate the semi-convergence point $k_0$, the estimated early stopping iterations do not deviate far from $k_0$, and the reconstructed solutions approximate well to the true solution. For pGKB\_HR with WGCV and SU, the relative error eventually decay flat as $k$ becomes sufficiently large, and slight higher than the best relative error of pGKB\_SPR at $k_0$. Although it is not shown here, in this experiment we find that both WGCV and SU over-estimate regularization parameters, and the estimate by WGCV is higher, thereby it computes a more over-smoothed solution.

\begin{figure}[htbp]
	\centering
	\subfloat[$\tau=0$]
	{\label{fig:4a}\includegraphics[width=0.42\textwidth]{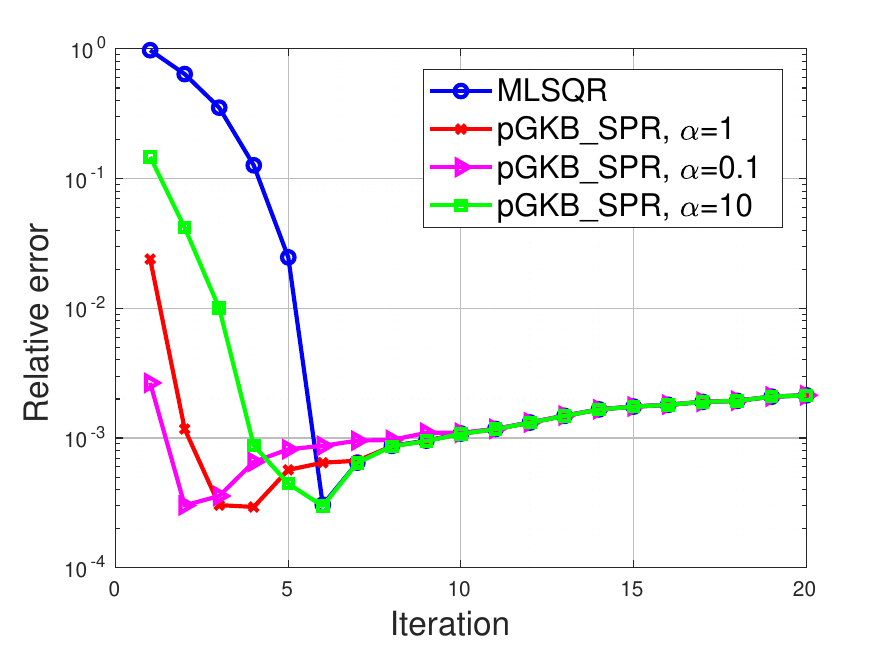}}
	\subfloat[$\alpha=1$]
	{\label{fig:4b}\includegraphics[width=0.42\textwidth]{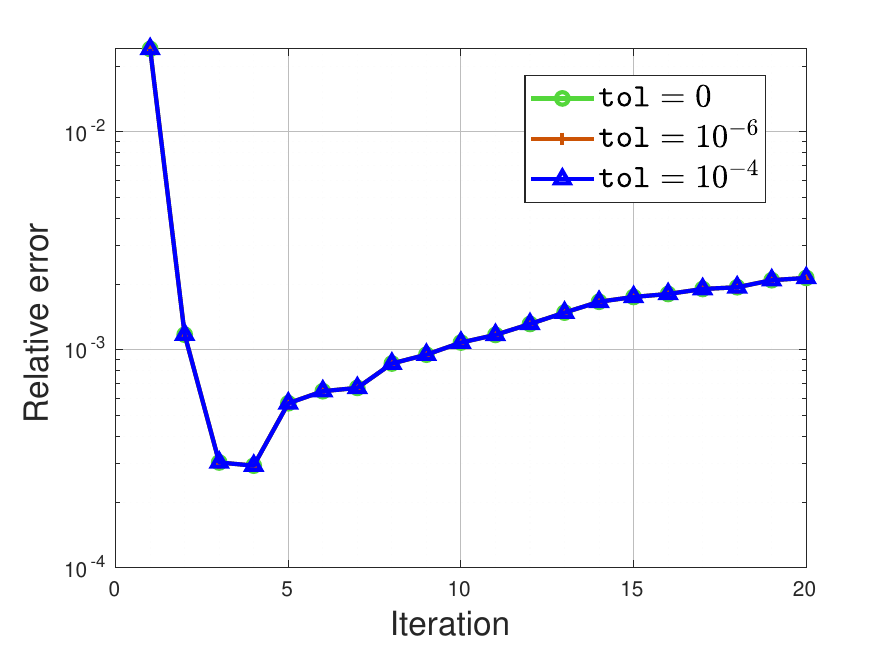}}
	\caption{Comparisons of semi-convergence curves between MLSQR and pGKB based subspace projection regularization algorithms for {\sf gauss1d}. (a) Comparison for different $\alpha$ where inner iterations are computed accurately. (b) Comparison for different solution accuracy of inner iterations.}
	\label{fig4}
\end{figure}

For {\sf gauss1d} with a piecewise constant $x_{\text{true}}$, the total variation regularization TV$(x)=\int_{\mathbb{R}}|\nabla x|\mathrm{d}t$ is the most suitable. Since TV$(x)$ is nonlinear, we construct $M$ to approximate TV$(x)$ at $x_{\mathrm{true}}$ by $x^TMx$ using the procedure in the LDFP method. We remark that this $M$ is only used for experimental purpose, while in practice we can not construct such a good $M$ since $x_{\mathrm{true}}$ in unknown. The LDFP method first replaces TV$(x)$ by TV$_{\beta}(x)=\int_{\mathbb{R}}\sqrt{|\nabla x|^2+\beta^2}\mathrm{d}t$ with $\beta$ a small positive value, and then linearize $\mathrm{TV}_{\beta}$ at $x$ using its gradient $L(x)$:
\begin{equation}
	L(x)y := -\nabla\cdot\left(\frac{1}{\sqrt{|\nabla x|^2+\beta^2}}\nabla y \right) .
\end{equation}
We choose $\beta=10^{-6}$ and construct $M$ by discretizing $L(x)$ at $x_{\text{true}}$ using the finite difference procedure. For details, see \cite{Vogel1996iterative,Chan1999convergence}. 

In this experiment, we first compare the pGKB method with MLSQR method proposed in \cite{Arridge2014}. Since the rank of $M$ is $n-1$, we let $M_{\delta}=M+\delta I$ with $\delta=10^{-6}$ and apply MLSQR to $\{A, M_{\delta}\}$. The semi-convergence curves are shown in \Cref{fig:4a}, where the inner iteration is computed accurately using matrix inversion, and the values of $\alpha$ are set as $1$, $0.01$ and $10$. We can find that pGKB\_SPR exhibit typical semi-convergence behavior, and the relative errors at the semi-convergence point for the three different $\alpha$ are almost the same as that of MLSQR. The impact of computing accuracy of inner iteration on the accuracy of regularized solution for pGKB\_SPR are shown in \Cref{fig:4b}, where we use $\alpha=1$. All the three curves almost coincide for many steps even after semi-convergence. This indicates that a much relaxed computing accuracy such as $\texttt{tol}=10^{-4}$ for inner iterations does not reduce accuracy of the final regularized solution.

\begin{figure}
	\centering
	\subfloat[Convergence history]{
		\begin{minipage}[b]{.5\linewidth}
			\includegraphics[width=2.5in]{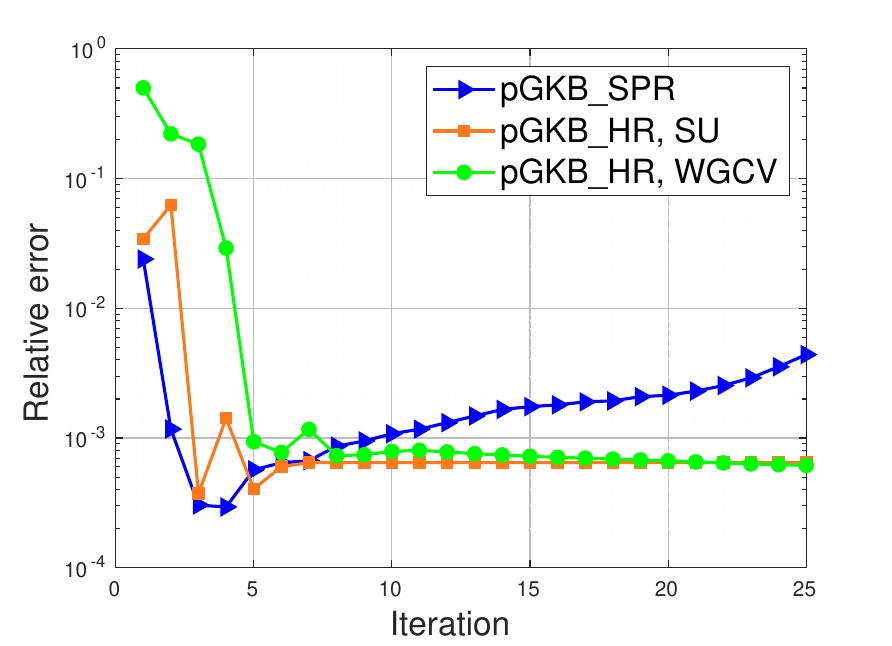}
		\end{minipage}
	} 
	\hspace{-10mm}
	\subfloat[Reconstructed solution]{
		\begin{minipage}[b]{.5\linewidth}
			\centering
			\includegraphics[width=1.33in]{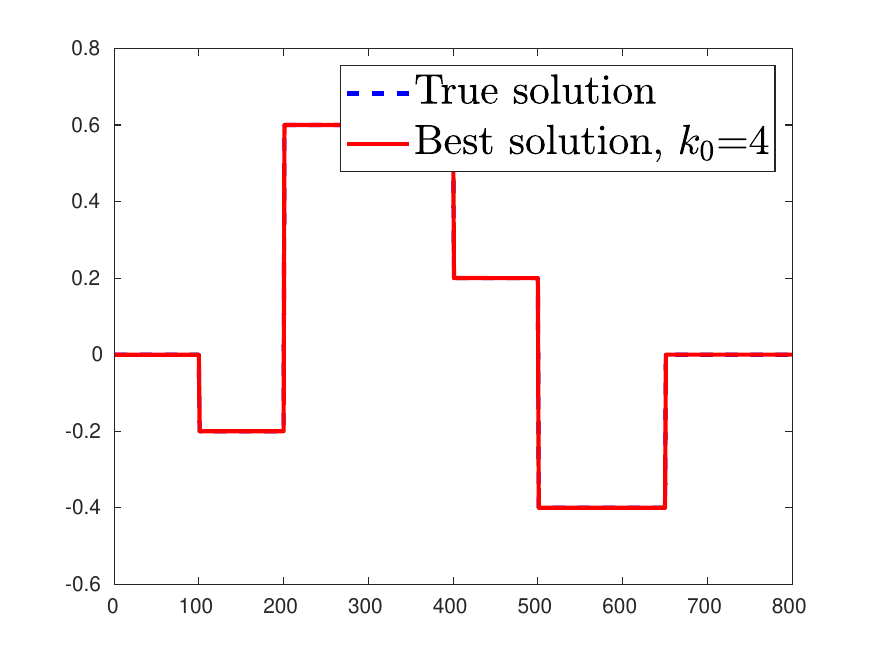}\hspace{-4mm}
			\includegraphics[width=1.33in]{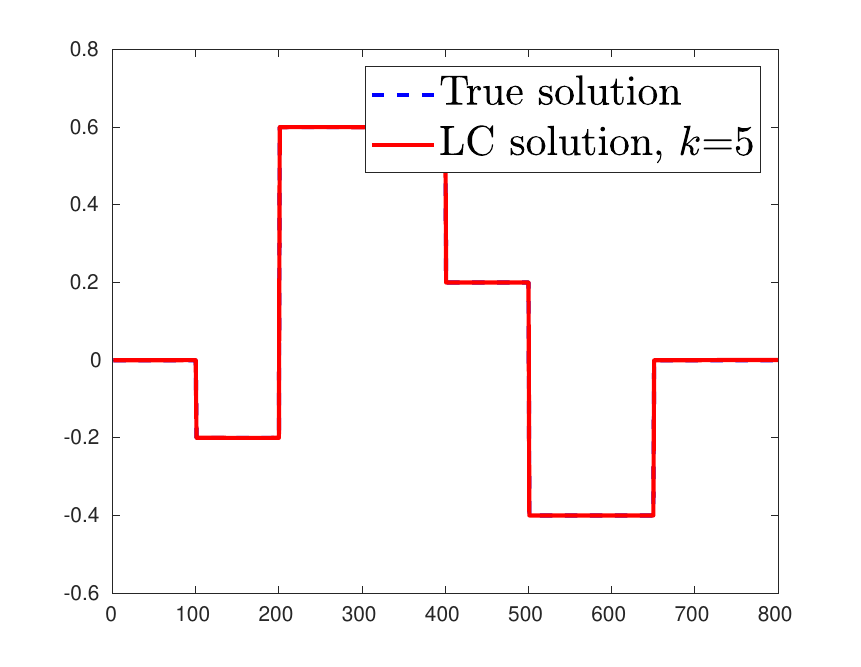}
			\vspace{-1.8mm}
			\includegraphics[width=1.33in]{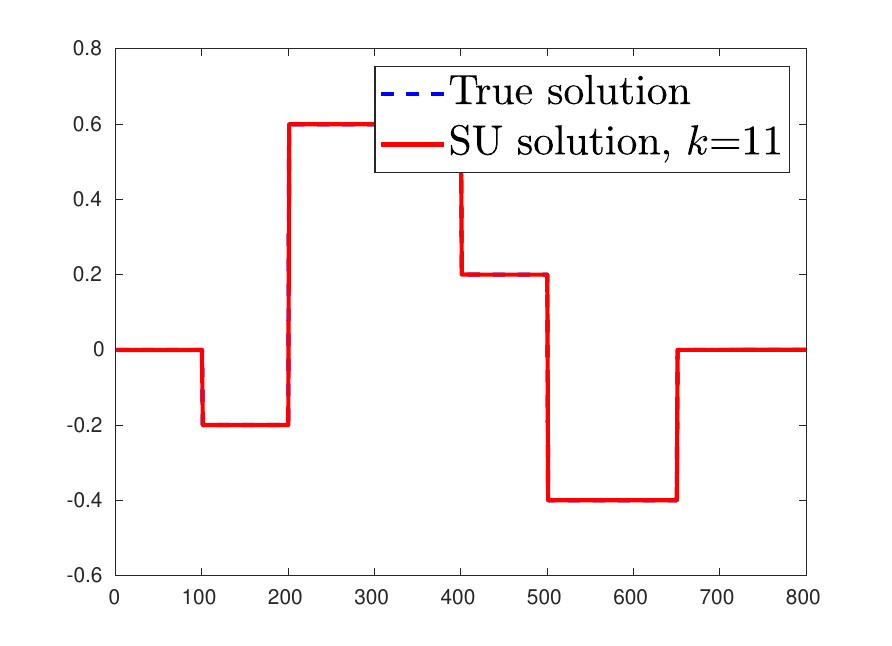}\hspace{-4mm}
			\includegraphics[width=1.33in]{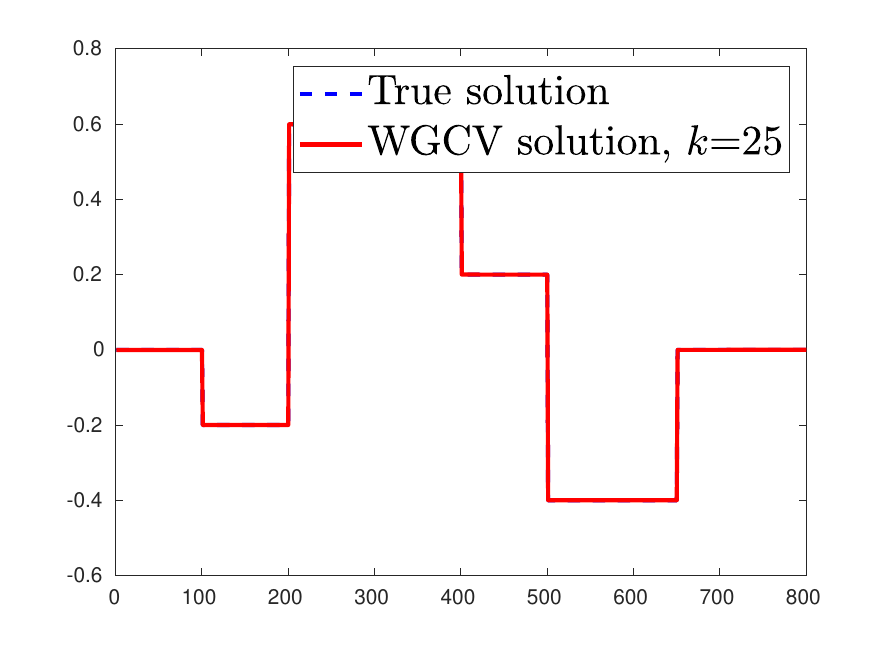}
		\end{minipage}
	}
	\caption{The relative error curves of regularized solutions computed by the pGKB based algorithms and corresponding reconstructed solutions for {\sf gauss1d}.}
	\label{fig5}
\end{figure}

\begin{table}[htp]
	\centering
	\caption{Relative errors of the final regularized solutions and corresponding early stopping iterations (in parentheses), where  $\varepsilon=5\times 10^{-4}$ for {\sf deriv2} and $\varepsilon=5\times 10^{-3}$ for {\sf gauss1d}.}
	\scalebox{0.73}{
	\begin{tabular}{llllll}
		\toprule
		Problem     & Best  & DP & LC  & SU & WGCV          \\
		\midrule
		{\sf deriv2} & 0.0064 (12) & 0.0087 (10) & 0.0120 (8) & 0.0105 (16) & 0.0165 (24)  \\
		{\sf gauss1d} & $2.2395\times 10^{-4}$ (4) & $3.0393\times 10^{-4}$ (3) & $5.6806\times 10^{-4}$ (5) & $6.4605\times 10^{-4}$ (11) & $6.1523\times 10^{-4}$ (25)  \\
		\bottomrule
	\end{tabular}}
	\label{tab1}
\end{table}

The convergence behavior of pGKB\_SPR and pGKB\_HR algorithms and the corresponding reconstructed 1D signals are shown in \Cref{fig5}, where $\alpha=1$ and $\texttt{tol}=10^{-6}$. The relative errors of the final regularized solutions and the corresponding iteration numbers are shown in \Cref{tab1}. For pGKB\_HR with WGCV and SU, the relative errors eventually decay flat as $k$ becomes sufficiently large and they are very close. Note that the iteration of WGCV does not satisfied the stopping criterion \cref{tol1} when the maximum iteration number $25$ are reached, since the default value for $\mathtt{tol1}$ is too small in this case. All the reconstructed signals approximate well to the original piecewise constant signal. This is due to the well-constructed regularizer matrix $M$ based on TV regularization and the incorporation of prior information about $x_{\text{true}}$ in the solution subspace $\mathrm{span}\{W_k\}$.

\subsection{Large-scale inverse problems}
The two large-scale test problems are chosen from \cite{Gazzola2019}. The first test problem is {\sf PRblurdefocus}, which models an image blurring problem caused by a spatially invariant out-of-focus blur. We use the true image `Hubble Space Telescope' with $256\times256$ pixels, and set the blur level as \texttt{`mild'} and use zero boundary condition to get $A\in\mathbb{R}^{256^2\times256^2}$. The second test problem is {\sf PRdiffusion}, which is a 2D inverse diffusion problem $\partial_t u=\nabla^{2}u$ in the domain $[0,T]\times [0,1]^2$ with Neumann boundary condition, and we aim to reconstruct the initial function $u_0$ from $u_{T}$. We set $T=0.005$ with $100$ time steps, and discretize the PDE on a $128\times128$ uniform finite element mesh to get the forward operator $A\in\mathbb{R}^{128^2\times128^2}$, which maps $x_{\text{true}}$ (the discretized $u_0$) to $b_{\text{true}}$ (the discretized $u_T$). The noise levels are set as $\varepsilon=0.002$ and $\varepsilon=0.001$ for {\sf PRblurdefocus} and {\sf PRdiffusion}, respectively. The true solutions and noisy observed data are shown in \Cref{fig6}.

\begin{figure}[htbp]
	\centering
	\subfloat
	{\label{fig:6a}\includegraphics[width=0.3\textwidth]{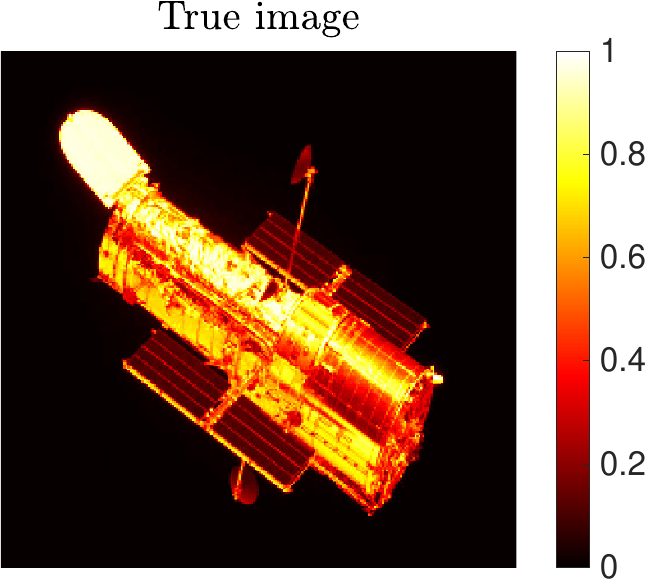}}\hspace{12mm}
	\subfloat
	{\label{fig:6b}\includegraphics[width=0.3\textwidth]{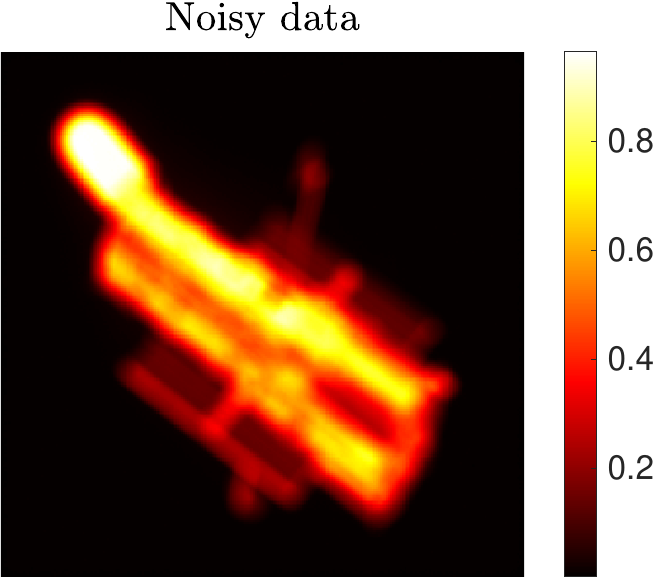}}
	\vspace{-3mm} 
	\subfloat
	{\label{fig:6c}\includegraphics[width=0.4\textwidth]{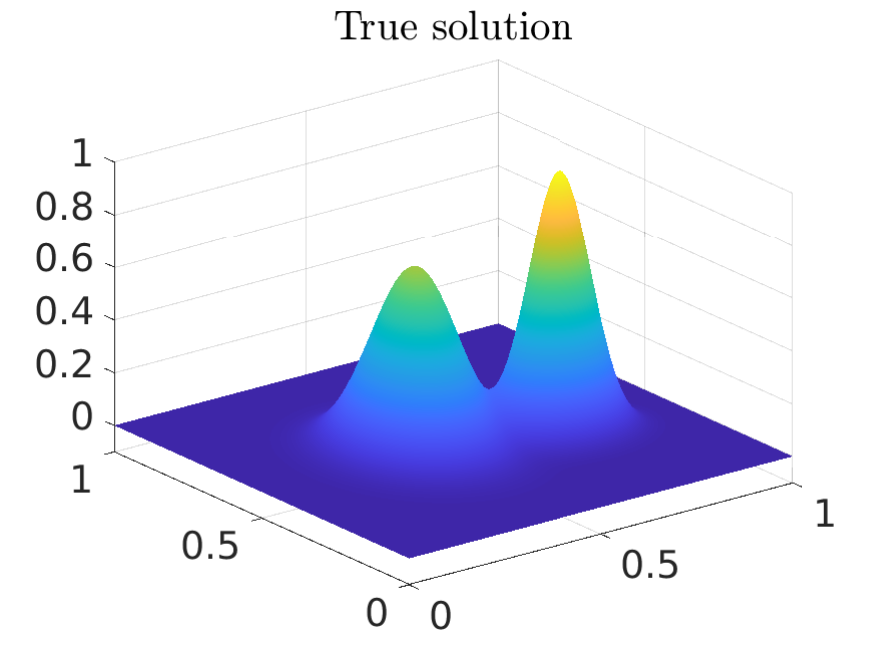}}
	\subfloat
	{\label{fig:6d}\includegraphics[width=0.4\textwidth]{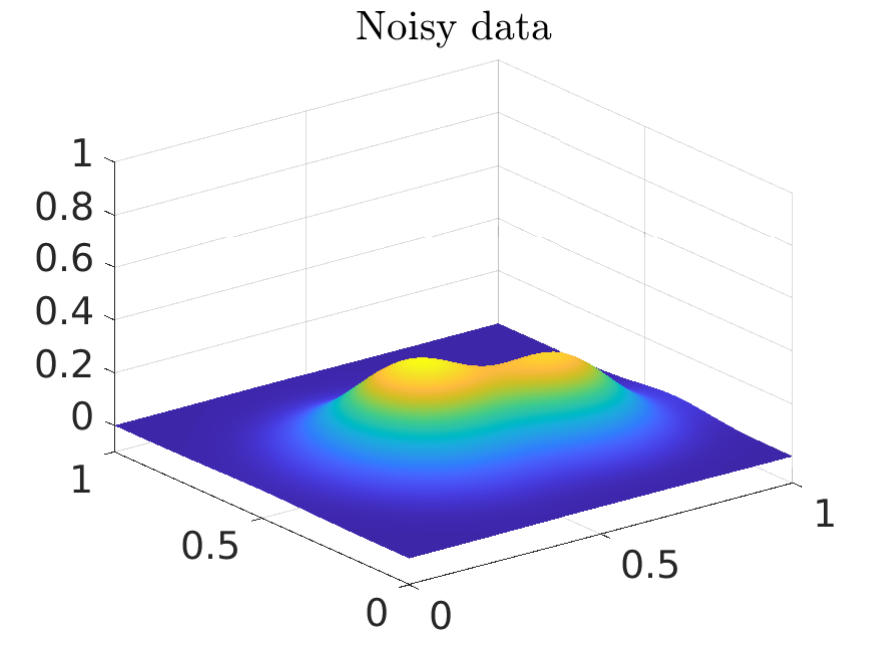}}
	\caption{Illustration of the true solution and noisy observed data. Top: true and noisy blurred image for {\sf PRblurdefocus}. Bottom: $x_{\text{true}}$ corresponding to $u_0$ and $b$ corresponding to noisy $u_T$ for {\sf PRdiffusion}.}
	\label{fig6}
\end{figure}

\begin{figure}
	\centering
	\subfloat[Convergence history]{\label{fig:7a}
		\begin{minipage}[b]{.5\linewidth}
			\includegraphics[width=2.5in]{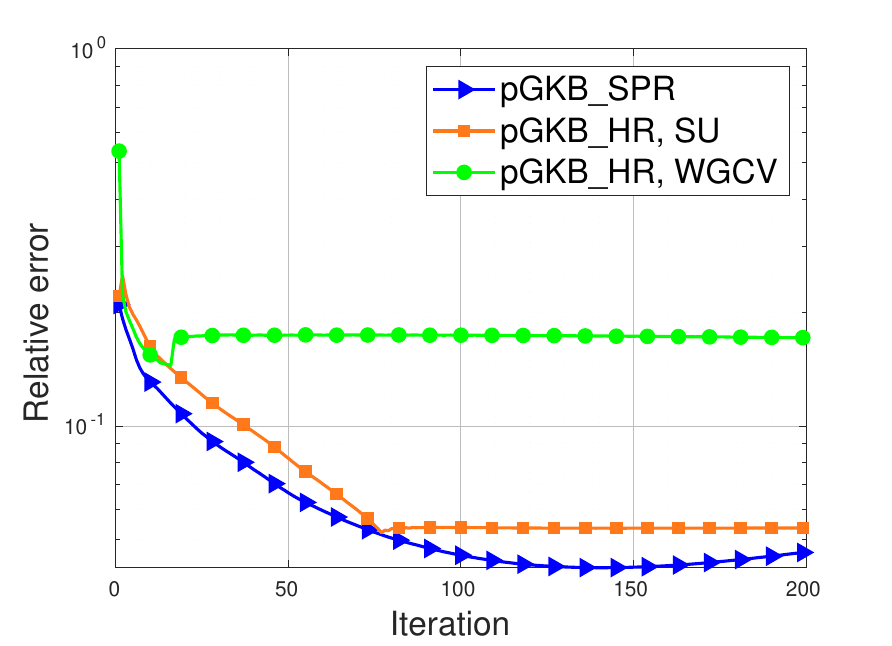}
		\end{minipage}
	} 
	\hspace{-8mm}
	\subfloat[Reconstructed solution]{\label{fig:7b}
		\begin{minipage}[b]{.5\linewidth}
			\centering
			\includegraphics[width=1.22in]{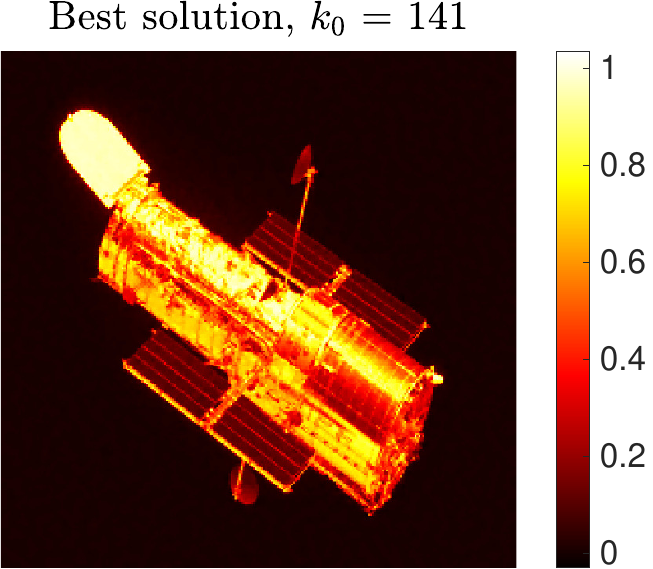}
			\includegraphics[width=1.22in]{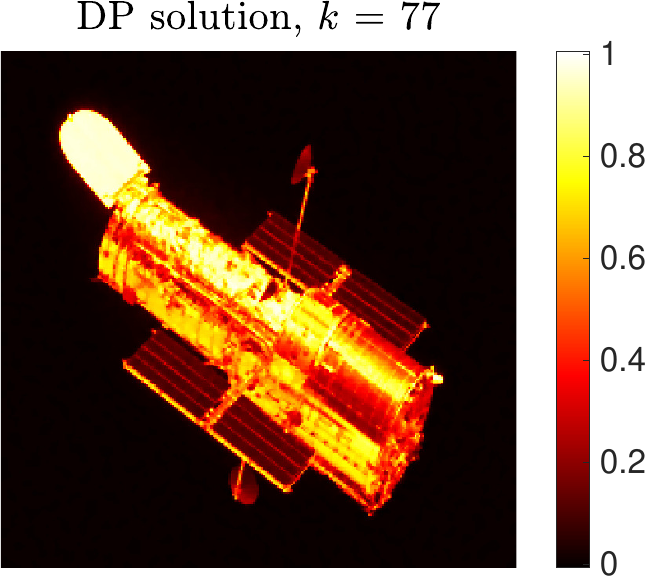}
			\quad
			\includegraphics[width=1.22in]{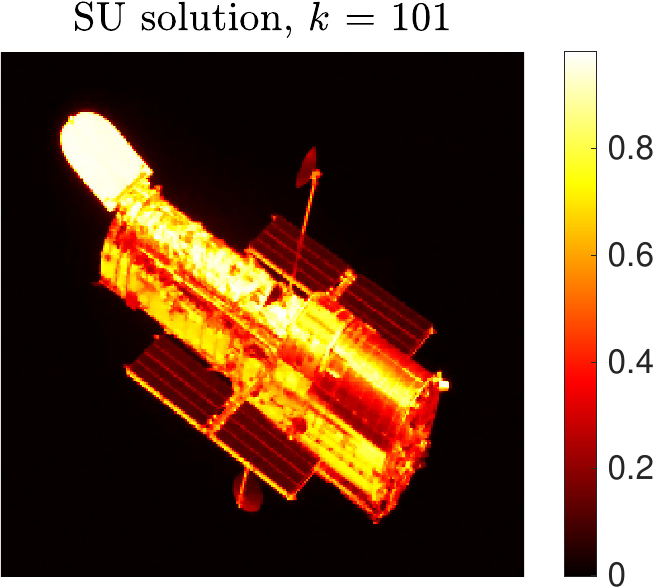}
			\includegraphics[width=1.22in]{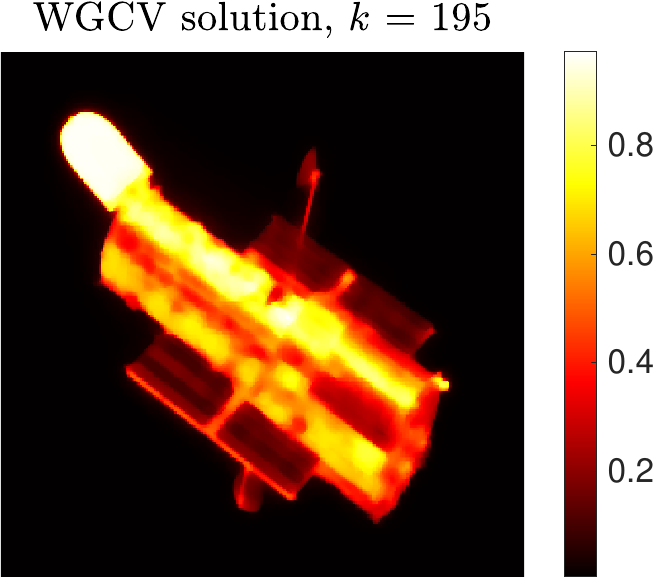}
		\end{minipage}
	}
	\caption{The relative error curves of regularized solutions computed by the pGKB based algorithms and corresponding reconstructed images for {\sf PRblurdefocus}.}
	\label{fig7}
\end{figure}

For {\sf PRblurdefocus}, we use the same procedure as for {\sf gauss1d} to construct $M$. Here the deblurring matrix $A$ is an object that only the matrix-vector products $Av$ or $A^Tv$ are available. Therefore, we can only compute the inner iteration by iteratively solving $Gs=A^Tu_i$. We set $\alpha=0.1$ and use \texttt{pcg.m} with stopping tolerance $\mathtt{tol}=10^{-6}$ to compute inner iterations. We compare the convergence behaviors of pGKB\_SPR and pGKB\_HR by plotting there convergence history curves. The convergence history curves and the corresponding reconstruct solutions are shown in \Cref{fig3}, and relative errors of the final regularized solutions and the corresponding iteration numbers are shown in \Cref{tab2}. From \Cref{fig:7a} we observe the typical semi-convergence behavior of pGKB\_SPR, and both DP and LC under-estimate the semi-convergence point $k_0$. For pGKB\_HR with SU, the relative error gradually decreases to a constant value for sufficient large $k$, and is is slightly higher than the relative error of pGKB\_SPR at $k_0$. For pGKB\_HR with WGCV, the relative error eventually stagnate at a value much higher than that of pGKB\_SPR at $k_0$, since WGCV significantly over-estimates the regularization parameter. The corresponding reconstructed images are shown in \Cref{fig:7b}, where we can clearly see that the reconstruction quality for WGCV is much poor than other algorithms. 

\begin{figure}
	\centering
	\subfloat[Convergence history]{\label{fig:8a}
		\begin{minipage}[b]{.45\linewidth}
			\includegraphics[width=2.5in]{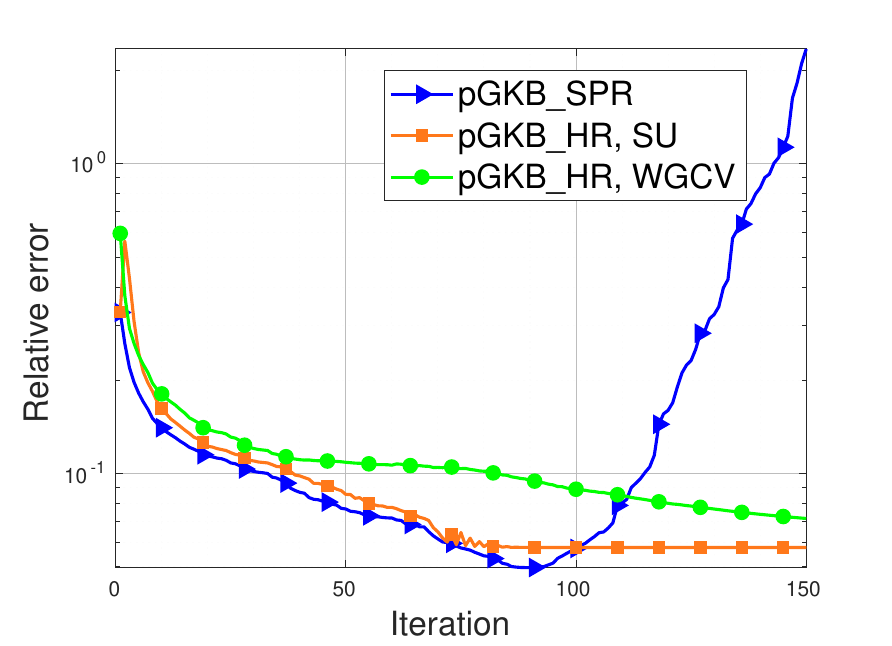}
		\end{minipage}
	} 
	\hspace{-6mm}
	\subfloat[Reconstructed solution]{\label{fig:8b}
		\begin{minipage}[b]{.55\linewidth}
			\centering
			\includegraphics[width=1.35in]{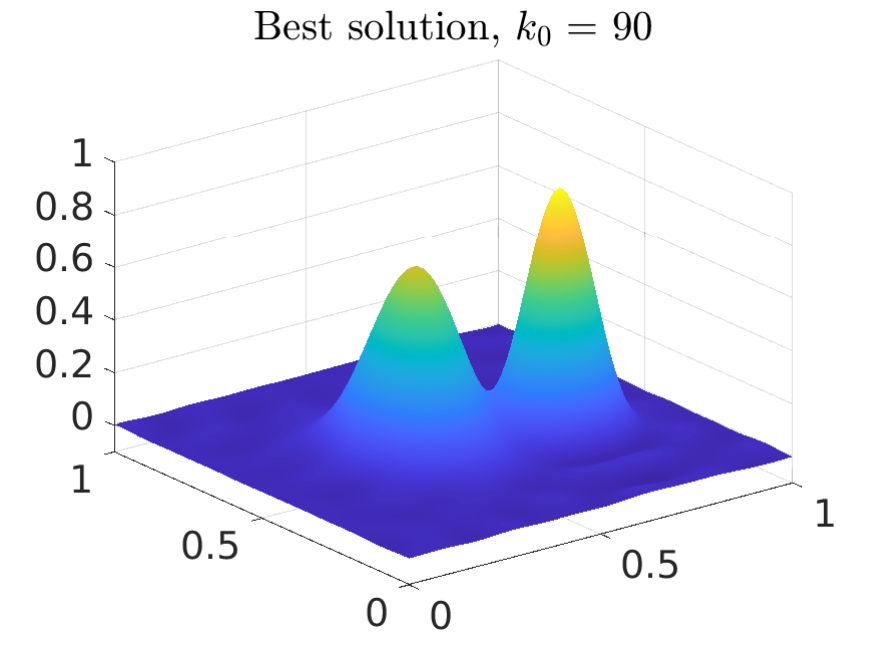}\hspace{-3mm}
			\includegraphics[width=1.35in]{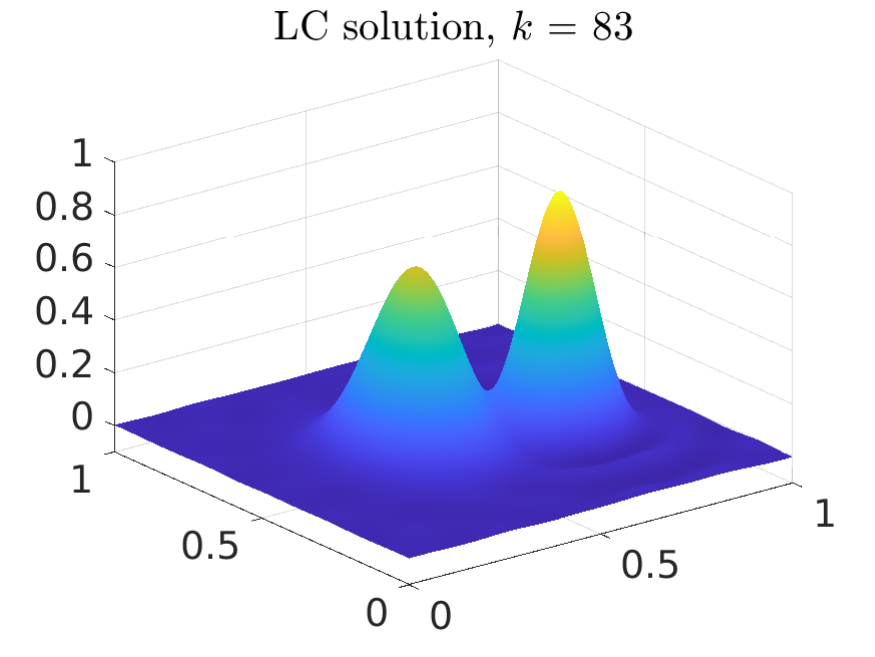}\vspace{-0.5mm}
			\includegraphics[width=1.35in]{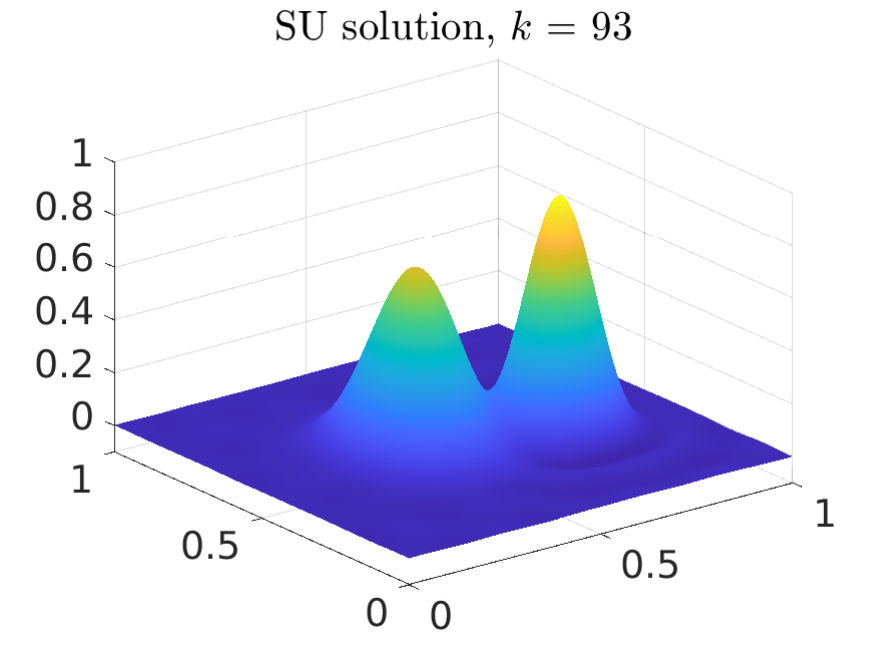}\hspace{-3mm}
			\includegraphics[width=1.35in]{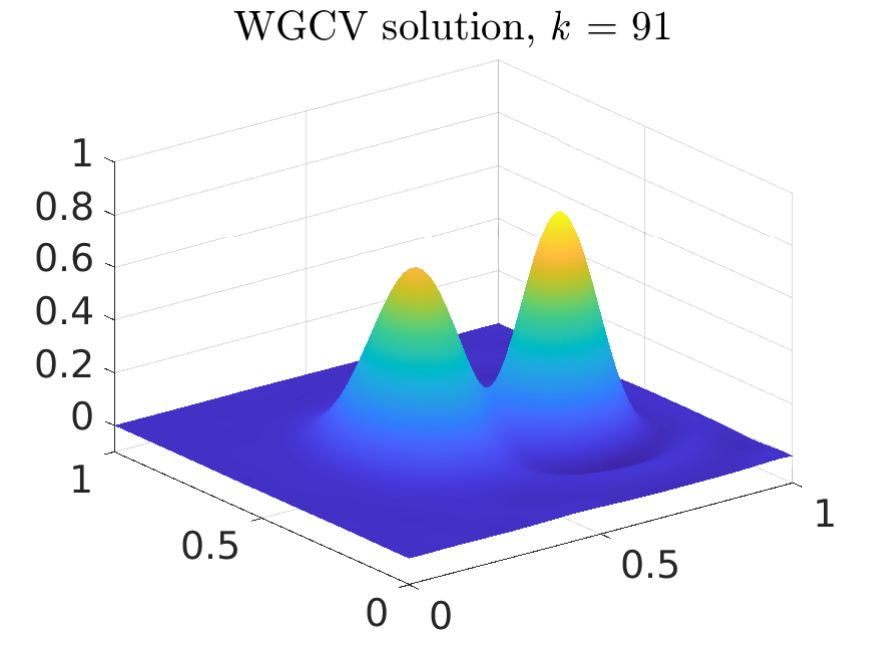}
		\end{minipage}
	}
	\caption{The relative error curves of regularized solutions computed by the pGKB based algorithms and corresponding reconstructed solutions for {\sf PRdiffusion}.}
	\label{fig8}
\end{figure}

\begin{table}[htp]
	\centering
	\caption{Relative errors of the final regularized solutions and corresponding early stopping iterations (in parentheses), where $\varepsilon=0.002$ for {\sf PRblurdefocus} and $\varepsilon=0.001$ for {\sf PRdiffusion}.}
	\scalebox{0.9}{
	\begin{tabular}{llllll}
		\toprule
		Problem     & Best  & DP & LC  & SU & WGCV          \\
		\midrule
		{\sf PRblurdefocus} & 0.0422 (141) & 0.0515 (77) & 0.0508 (79) & 0.0539 (101) & 0.1717 (195)  \\
		{\sf PRdiffusion} & 0.0497 (90) & 0.0602 (71) & 0.0526 (83) & 0.0578 (93) & 0.0946 (91)  \\
		\bottomrule
	\end{tabular}}
	\label{tab2}
\end{table}

For {\sf PRdiffusion}, we set $M$ as the discretized 2D negative Laplacian to enforce some smoothness on the desired initial $u_0$. Here the forward operator matrix $A$ is a functional handle that represents the computation process of numerical solution of $\partial_{t}u=\nabla^{2}u$. Therefore, we can only compute the inner iteration by solving $Gs=A^Tu_i$ using an iterative solver that is based on matrix-vector products. We set $\alpha=1$ in pGKB and use \texttt{pcg.m} with stopping tolerance $\mathtt{tol}=10^{-6}$ to compute inner iterations. From \Cref{fig8} and \Cref{tab2} we find that both DP and LC slighly under-estimate $k_0$ for pGKB\_SPR, and both the corresponding reconstructed solutions are of high quality. For pGKB\_HR, as $k$ becomes sufficient large, both WGCV and SU compute iterative solutions with relative errors decreasing toward a value slightly larger than the best for pGKB\_SPR, and the corresponding final regularized solutions approximate well to $x_{\text{true}}$.

\section{Conclusion and outlook}\label{sec7}
For linear inverse problems with general-form Tikhonov regularization term $x^TMx$, where $M$ is positive semi-definite, we have proposed several iterative regularization algorithms. These algorithms base upon a new iterative process called the preconditioned Golub-Kahan bidiagonalization (pGKB) that implicitly utilizes a proper preconditioner to generate solution subspaces incorporating prior properties of the desired solution. The pGKB based subspace projection regularization (pGKB\_SPR) algorithm is proposed, where the discrepancy principle or L-curve is used as an early stopping criterion. The regularization effect of pGKB\_SPR is analyzed by showing that the iterative solution has a filtered GSVD expansion form, thus revealing the semi-convergence behavior of it. To avoid semi-convergence of pGKB\_SPR, two pGKB based hybrid regularization (pGKB\_HR) algorithms are proposed that adopt WGCV and secant update for determining regularization parameters at each iteration, respectively. Both small-scale and large-scale linear inverse problems are used to test the proposed algorithms and illustrate excellent effectiveness and performance of them.

There are everal issues remaining to be further studied. For example, numerical results indicate that the accuracy of inner iteration in pGKB may be moderately relaxed without compromising the accuracy of final regularized solution, raising the need for theoretical analysis about the required accuracy of inner iteration. Another issue concerns the secant update method for pGKB\_HR, which numerically exhibits good convergence for iterative solution. Thus it is neccessary to analyze the convergence behaviors of regularization parameter $\mu_k$ and iterative solution.

\section*{Acknowledgments}
The author thank Dr. Long Wang and Prof. Weile Jia for their consistent encouragement and support during the research.

\bibliographystyle{siamplain}
\bibliography{references}

\end{document}